\documentclass[12pt]{article}

\usepackage{geometry}
\usepackage{hyperref} 

\usepackage{amsmath}
\usepackage{amssymb}
\usepackage{amscd}
\usepackage{textcomp}
\usepackage{dsfont}
\usepackage{mathrsfs}

\long\def\forget#1{}

\usepackage{color}
\newcounter{commentcounter}
\newcommand{\comment}[1]{\stepcounter{commentcounter}{\color{red}\textbf{Comment \arabic{commentcounter}.} #1}
\immediate\write16{}
\immediate\write16{Warning: There was still a comment . . . }
\immediate\write16{}}

\newcounter{urscommentcounter}

\def\?{\ 
{\bf\color{red}???}\ 
\immediate\write16{}
\immediate\write16{Warning: There was still a question mark . . . }
\immediate\write16{}}

\usepackage{amsthm}
\theoremstyle{plain}
\newtheorem{theorem}{Theorem}[section]

\newtheorem{lemma}[theorem]{Lemma}

\newtheorem{corollary}[theorem]{Corollary}
\newtheorem{proposition}[theorem]{Proposition}

\theoremstyle{definition}
\newtheorem{definition}[theorem]{Definition}
\newtheorem{definition-theorem}[theorem]{Definition-Theorem}
\newtheorem{definition-remark}[theorem]{Definition-Remark}

\newtheorem{remark}[theorem]{Remark}

\theoremstyle{remark}

\usepackage{xy}
\xyoption{all}

\newdir^{ (}{{}*!/-3pt/\dir^{(}}    
\newdir^{  }{{}*!/-3pt/\dir^{}}    
\newdir_{ (}{{}*!/-3pt/\dir_{(}}    
\newdir_{  }{{}*!/-3pt/\dir_{}}

\newcounter{zahl}

\def\theenumi{(\alph{enumi})}

\def\p@enumii{\theenumi}

\newcommand{\DS}{\displaystyle}

\newcommand{\SC}{\scriptstyle}

\newcommand{\cM}{\mathcal{M}}
\newcommand{\cO}{\mathcal{O}}

\DeclareMathOperator{\Aut}{Aut}

\DeclareMathOperator{\End}{End}

\DeclareMathOperator{\Frob}{Frob}
\DeclareMathOperator{\Gal}{Gal}
\DeclareMathOperator{\GL}{GL}

\DeclareMathOperator{\Koh}{H}

\DeclareMathOperator{\Hom}{Hom}

\DeclareMathOperator{\QEnd}{QEnd}
\DeclareMathOperator{\QHom}{QHom}
\DeclareMathOperator{\Quot}{Frac}

\DeclareMathOperator{\Rep}{Rep}

\DeclareMathOperator{\SL}{SL}

\DeclareMathOperator{\Spec}{Spec}
\DeclareMathOperator{\Spf}{Spf}

\newcommand{\alg}{{\rm alg}}

\DeclareMathOperator{\coker}{coker}

\newcommand{\fppf}{{\it fppf\/}}
\newcommand{\fpqc}{{\it fpqc\/}}

\DeclareMathOperator{\id}{\,id}
\DeclareMathOperator{\im}{im}
\DeclareMathOperator{\coim}{coim}

\renewcommand{\mod}{{\rm\,mod\,}}

\DeclareMathOperator{\Frac}{Frac}

\DeclareMathOperator{\rank}{rank}

\DeclareMathOperator{\rk}{rk}
\newcommand{\sep}{{\rm sep}}

\renewcommand{\phi}{\varphi}
\renewcommand{\epsilon}{\varepsilon}
\usepackage{amsfonts}
\newcommand{\BOne} {{\mathchoice{\hbox{\rm1\kern-2.7pt l\kern.9pt}}
                              {\hbox{\rm1\kern-2.7pt l\kern.9pt}}
                              {\hbox{\scriptsize\rm1\kern-2.3pt l\kern.4pt}}
                              {\hbox{\scriptsize\rm1\kern-2.4pt l\kern.5pt}}}}

\newcommand{\BA}{{\mathbb{A}}}

\newcommand{\BC}{{\mathbb{C}}}
\newcommand{\BD}{{\mathbb{D}}}

\newcommand{\BF}{{\mathbb{F}}}
\newcommand{\BG}{{\mathbb{G}}}

\newcommand{\BL}{{\mathbb{L}}}

\newcommand{\BN}{{\mathbb{N}}}

\newcommand{\BP}{{\mathbb{P}}}
\newcommand{\BQ}{{\mathbb{Q}}}

\newcommand{\BZ}{{\mathbb{Z}}}

\newcommand{\CA}{{\cal{A}}}

\newcommand{\CC}{{\cal{C}}}

\newcommand{\CF}{{\cal{F}}}
\newcommand{\CG}{{\cal{G}}}
\newcommand{\CH}{{\cal{H}}}
\newcommand{\CI}{{\cal{I}}}

\newcommand{\CM}{{\cal{M}}}
\newcommand{\CN}{{\cal{N}}}
\newcommand{\CO}{{\cal{O}}}

\newcommand{\CT}{{\cal{T}}}

\newcommand{\CV}{{\cal{V}}}

\newcommand{\CZ}{{\cal{Z}}}

\newcommand{\FG}{{\mathfrak{G}}}

\newcommand{\FJ}{{\mathfrak{J}}}

\newcommand{\FP}{{\mathfrak{P}}}

\newcommand{\Fa}{{\mathfrak{a}}}

\newcommand{\scrH}{{\mathscr{H}}}

\newcommand{\mot}{{\cM ot_C^{\ul \nu}}}

\let\setminus\smallsetminus

\newcommand{\ul}[1]{{\underline{#1}}}
\newcommand{\ol}[1]{{\overline{#1}}}
\newcommand{\wh}[1]{{\widehat{#1}}}

\usepackage{ifthen}

\newcommand{\invlim}[1][]{\ifthenelse{\equal{#1}{}}% falls Argument leer
{\DS \lim_{\longleftarrow}}%                         verwende niedrige Version
{\DS \lim_{\underset{#1}{\longleftarrow}}}%  sonst:  verwende Argument
}

\newcommand{\dirlim}[1][]{\ifthenelse{\equal{#1}{}}% falls Argument leer
{\DS \lim_{\longrightarrow}}%                        verwende niedrige Version
{\DS \lim_{\underset{#1}{\longrightarrow}}}% sonst:  verwende Argument
}

\newcommand{\dbl}{{\mathchoice{\mbox{\rm [\hspace{-0.15em}[}}
                              {\mbox{\rm [\hspace{-0.15em}[}}
                              {\mbox{\scriptsize\rm [\hspace{-0.15em}[}}
                              {\mbox{\tiny\rm [\hspace{-0.15em}[}}}}
\newcommand{\dbr}{{\mathchoice{\mbox{\rm ]\hspace{-0.15em}]}}
                              {\mbox{\rm ]\hspace{-0.15em}]}}
                              {\mbox{\scriptsize\rm ]\hspace{-0.15em}]}}
                              {\mbox{\tiny\rm ]\hspace{-0.15em}]}}}}
\newcommand{\dpl}{{\mathchoice{\mbox{\rm (\hspace{-0.15em}(}}
                              {\mbox{\rm (\hspace{-0.15em}(}}
                              {\mbox{\scriptsize\rm (\hspace{-0.15em}(}}
                              {\mbox{\tiny\rm (\hspace{-0.15em}(}}}}
\newcommand{\dpr}{{\mathchoice{\mbox{\rm )\hspace{-0.15em})}}
                              {\mbox{\rm )\hspace{-0.15em})}}
                              {\mbox{\scriptsize\rm )\hspace{-0.15em})}}
                              {\mbox{\tiny\rm )\hspace{-0.15em})}}}}

\newcommand{\dotBD}{\vbox{\hbox{\kern2pt\bf.}\vskip-4.5pt\hbox{$\BD$}}}

\DeclareMathOperator{\Nilp}{\CN \!{\it ilp}}

\def\ulM{{\underline{M\!}\,}{}}

\def\s{\sigma^\ast}

\def\longto{\longrightarrow}
\def\isoto{\stackrel{}{\mbox{\hspace{1mm}\raisebox{+1.4mm}{$\SC\sim$}\hspace{-3.5mm}$\longrightarrow$}}}

\newbox\mybox
\def\arrover#1{\mathrel{
       \setbox\mybox=\hbox spread 1.4em{\hfil$\scriptstyle#1$\hfil}
       \vbox{\offinterlineskip\copy\mybox
             \hbox to\wd\mybox{\rightarrowfill}}}}

\newcommand{\BaseOfD}{\BF}

\newcommand{\genericG}{P}

\newcommand{\Vect}{V\!ect}

\newcommand{\tauGlob}{\tau}

\newcommand{\charsect}{s}

\begin{document}

%%%%%%%%%%%%%%%%%%%%%%%%%%%%%%%%%%%%%%%%%%%%%%%%%%%%%%%%%%%%%%%%%%%%%%

\date{\today}

\title{Category Of $C$-Motives\\ Over
Finite Fields}

\author{Esmail Arasteh Rad and Urs Hartl}

\maketitle

\begin{abstract}

In this article we introduce and study a motivic category in the arithmetic of function fields, namely the category of motives over an algebraic closure $L$ of a finite field with coefficients in a global function field over this finite field. It is semi-simple, non-neutral Tannakian and possesses all the expected fiber functors. This category generalizes the previous construction due to Anderson and is more relevant for applications to the theory of $G$-Shtukas, such as formulating the analog of the Langlands-Rapoport conjecture over function fields. We further develop the analogy with the category of motives over $L$ with coefficients in $\BQ$ for which the existence of the expected fiber functors depends on famous unproven conjectures.

\end{abstract}

%\tableofcontents 

%%%%%%%%%%%%%%%%%%%%%%%%%%%%%%%%%%%%%%%%%%%%%%%%%%%%%%%%%%%%%%%
%
%       Introduction
%
%%%%%%%%%%%%%%%%%%%%%%%%%%%%%%%%%%%%%%%%%%%%%%%%%%%%%%%%%%%%%%%

\section{Introduction}

Let $\BF_q$ be a finite field with $q$ elements, let $C$ be a smooth, projective, geometrically connected curve over $\BF_q$ with function field $Q=\BF_q(C)$, and let $\FG$ be a flat affine group scheme of finite type over $C$. The moduli stacks of global $\FG$-shtukas appear as function field analogs for Shimura varieties; for some explanation regarding this analogy see for example \cite[Chapter 3]{Ara}.
There are several approaches to study the geometry of these moduli stacks. In \cite{AH_Local}, \cite[Chapter 4]{AH_LM} and \cite{AH_Global} the authors developed an approach which is based on the relation between these moduli stacks and certain moduli spaces for local $\BP$-shtukas. Namely, in \cite{AH_Local} we developed the local theory of global $\FG$-shtukas. In particular, for smooth group scheme $\FG$ over $C$ with reductive connected generic fiber, we constructed Rapoport-Zink spaces for local $\BP$-shtukas, where $\BP$ is the base change of $\FG$ to the complete local ring at a point of $C$. This construction generalizes the previous construction in \cite{H-V} for constant groups $\FG=G\times_{\BF_q} \Spec \BF_q\dbl z\dbr$, where $G$ is split reductive over $\BF_q$, to the ramified case. Furthermore in \cite[Chapter 4]{AH_LM} the authors use the deformation theory of local $\BP$-shtukas, investigated in \cite{AH_Local}, to establish the analogue of Rapoport-Zink theory of local models for moduli of $\FG$-shtukas. On the other hand, using Rapoport-Zink spaces for local $\BP$-shtukas we established the uniformization theory of the moduli stacks of global $\FG$-shtukas in \cite{AH_Global}. Based on these results we investigate the analog of the Langlands-Rapoport conjecture over function fields in \cite{AH_LR}. To this aim, as we will discuss in this article, one has to introduce and study the properties, of a relevant motivic category $\mot(L)$ over a finite field extension $L$ of $\BF_q$, as well as the associated fundamental motivic groupoid. Here $\ul\nu$ denotes an n-tuple $(\nu_1,\ldots,\nu_n)$ of pairwise different closed points of $C$. In this context, our initial motivation for studying this category was the fact that it provides a tool to describe the quasi-isogeny classes of $\FG$-shtukas in terms of the representations of the corresponding motivic groupoid.\\ 
Apart from what we mentioned above, it seems to us that the category $\mot(L)$ deserves to be studied for its own as well. Note that it basically arises as a natural generalization of previous constructions of the category of motives over function fields considered by various authors; for example see \cite{Anderson}, \cite{TaelmanI} and \cite{Bor-Har1}. In addition let us mention that in another research project we aim to clarify the possible relation with the Galois gerbe introduced by Kottwiz in \cite{Kot}.\\

Let us now briefly explain how the category $\mot(L)$ naturally arises in the context of motives over function fields. In \cite{TaelmanI}, Taelman proposes several categories that serve the analogous role (of the Grothendieck category of Chow motives) over function fields. To this goal, he first considers the Anderson category of $A$-motives \cite{Anderson}, for $A=\BF_q\dbl z\dbr$, then after extending the coefficients to rational coefficients (by tensoring up the Hom-sets with $Q=\Frac(A)$), he formally inverts (tensor powers of) the Carlitz motive $\CC$.
The resulting category $t\CM^\circ$ together with the obvious fiber functor $\omega: t\CM^\circ\to Q-vector~spaces$ provides a tannakian category which is a candidate for the analogous motivic category over function fields. Still one may naturally want
\begin{enumerate}

\item[-]
to consider motives which admit  multiplications by a Dedekind domain which is strictly bigger than $\BF_q[z]$,  
\item[-]
to construct a category analogous to the category (mixed) motives over a general base, and
\item[-] to geometrize this category. More precisely, one may think of Shimura varieties as a moduli for motives according to the Deligne's conception of Shimura varieties \cite{Deligne2}. But the above category of motives over function fields do not behave well concerning moduli problems. This is for example because there are too much freedom at infinity of $\Spec A$.

\end{enumerate}

To handle the above issues,  one makes Anderson $A$-motives completed at the place infinity $\infty$ of a curve $C$, with $A:=\Gamma(\dot{C},\CO_C)$, in the following sense. Namely, we replace the (locally) free $A_L$-module $M$ by a locally free sheaf $\CM$ of $\CO_{C_L}$-modules (or equivalently a vector bundle over $C_L$).  Note that one should be a bit careful with extending $\tau$ over infinity. Recall that, according to the definition of $A$-motives, one requires the morphism $\tau$ to have it's zeros along $V(J)$, where $J$ is the ideal corresponding to the graph of a (characteristic) section $s_1: S:=\Spec L\to \dot{C}=\Spec A$. This in particular indicates that the morphism $\tau$ can not be defined over the whole (relative) curve $C_S$. This is because over a projective curve, the balance between order of zero's and poles of $\tau$ should be preserved. Therefore to provide an appropriate definition, one should allow further characteristic section(s) $s_i$. Note that aside from making the definition more efficient, in fact, introducing several characteristics turns out to be an extremely useful tool in function fields set up, which is still absent over number fields. Note in addition that introducing further characteristic sections corresponds to inverting several Carlitz-Hayes motives, see \cite{HJ},  in Taelman's construction. Regarding this, one eventually comes to the Definition \ref{DefCatCMotives} of the $Q$-linear category $\mot(L)$  whose objects are $C$-motives $\ul\CM=(\CM,\tau_\CM)$ of characteristic $\ul\nu$, with quasi-morphisms as its morphisms; see Definition \ref{DefCatCMotives} b).\\

Note that sending a $C$-motive $(\CM,\tau_\CM)$ to the generic fiber $\CM_\eta$ equips the category with a fiber functor 
$$
\omega:\mot(L)\longto Q_L\text{-vector spaces}.
$$    

Let us now briefly explain the content of this article. In Section~\ref{SubSectRealizationFunctors} we construct on $\mot (L)$ \'etale realization functors $\omega^\nu_{Q_\nu}$ when $\nu\notin\ul\nu$ and crystalline realization functors $\Gamma_\nu$ when $\nu\in\ul\nu$. Let $\nu\notin\ul\nu$ be a closed point of $C$ away from the characcteristic places. For a $C$-motive $\ul\CM$ over $L$ the \'etale realization $\omega^\nu_{Q_\nu}(\ul\CM)$ is a vector space over the $v$-adic completion $Q_v$ of $Q$ equiped with a continuous action of the Galois group $\Gamma_L=\Gal(\ol\BF_q/L)$. This allows to formulate the analog of the Tate conjecture for this category:

\begin{theorem}
Let $\ul\CM$ and $\ul\CM'$ be $C$-motives over a finite field $L$. Then there are isomorphisms

$$
\QHom_L(\ul\CM,\ul\CM')\otimes_Q Q_\nu \tilde{\longrightarrow} \Hom_{{Q_\nu}[\Gamma_L]}(\omega^\nu_{Q_\nu}(\ul\CM),\omega^\nu_{Q_\nu}(\ul\CM')) 
$$

of $Q_\nu$-vector spaces. Moreover if $\ul\CM=\ul\CM'$ then the above are ring isomorphisms. 
\end{theorem}
\noindent
This appears as Theorem \ref{Thm_Tate_Conj} in the text. \\ 

Our method for studying this category is  almost elementary and similar to the number fields set up. Namely, by looking at the endomorphism algebra associated with the objects, e.g. we observe that

\begin{theorem}\label{Thm_SemisimplicityQEnd}

Let $L/\BF_q$ be a finite field extension. Let $\ul\CM$ be a $C$-motive over $L$ and let $\pi\in\QEnd_L(\ul\CM)$ be its $L$-Frobenius. Let $\nu\notin\ul\nu$ be a closed point of $C$ away from the characcteristic places. The following statements are equivalent

\begin{enumerate}

\item
$\ul\CM\in\mot(L)$ is semi-simple. 

\item
$E:=\QEnd_L(\ul\CM)$ is semi-simple. 

\item

$\End_{Q_\nu[\Gamma_L]}(\omega_{Q_\nu}^\nu(\ul\CM))$ is semisimple. 

\item

$F_\nu:=Q_\nu[\pi_\nu]$ is semisimple, where $\pi_\nu=\omega_{Q_\nu}^\nu(\pi)\in \End_{Q_\nu}(\omega_{Q_\nu}^\nu(\ul\CM))$. 

\item

$F:=Q[\pi]$ is semi-simple.

%If $\ul\CM\in\mot(L)$ is semi-simple then the above functor sends the full subcategory of semisimple motives to semisimple $Q$-algebras. Furthermore $E$ is semisimple if and only if $E_\nu$ is semisimple.

\end{enumerate}

\end{theorem}
\noindent
This is Theorem \ref{Thm_SemisimplicityQEnd} in the text.\\

%\noindent
%We study the endomorphism algebra $\QEnd_L(\ul\CM)$ further and in particular we prove the following proposition. This appears as proposition \ref{Proph|r} and Proposition \ref{Prop_ExtremeCases} in the text.

We study the endomorphism algebra $\QEnd_L(\ul\CM)$ further and in particular we prove proposition \ref{Proph|r} and Proposition \ref{Prop_ExtremeCases}.\\

As we mentioned earlier the notion of $C$-motives is relevant for $\FG$-Shtukas; we discuss this in Section \ref{SectGSht}. In particular we introduce the category of $G$-$C$-motives, see Definition \ref{DefGMotives}, and we explain its relation to the category of $\FG$-shtukas. In Proposition  \ref{Prop_GeomMixedareMixedOverAFiniteExtensionII} we prove a finiteness theorem for $G$-$C$-motives that arise from $\FG$-shtukas.\\
In section \ref{Subsec_Global-Local Functor} we recall the construction of the global-local functor and furthermore, as an additional analogy with the theory of abelian varieties, we see that for a fixed $C$-motive $\ul\CM$ in $\mot(L)$, one can express the set of quasi-isogenies $\phi:\ul\CM' \to \CM$ in terms of Galois stable sublattices in the \'etale realization of $\ul\CM$ at the place $\nu$. See Corollary \ref{Cor_Qiso&SL}.

\begin{proposition}
Let $\phi:\ul\CM' \to \ul\CM$ be a quasi-isogeny of $C$-motives in $\mot (L)$. Then $\omega_Q^\nu(\phi)$ identifies $\omega^\nu(\ul\CM')$ with a $\Gamma_L$-stable sublattice of $\omega_Q^\nu(\ul\CM)$. This gives a one to one correspondence between  the following sets

$$
\{ \text{quasi-isogenies $\ul\CM'\to\ul\CM$ in $\mot(L)$ which are isomorphisms above $\ul \nu$}\}  
$$

and

$$
\{\text{$\Gamma_L$-stable sublattice $\Lambda_\nu \subseteq \omega_Q^\nu(\ul\CM)$ which are contained in $\omega^\nu(\ul\CM)$} \}.
$$

\end{proposition}

Note that $\mot(L)$ consists of mixed motives
rather than pure ones, and to define the subcategory of pure objects one has to impose purity conditions. But nevertheless, as we will see in Section \ref{Sect_Quasi-isogeny classes and Honda-Tate theory}, a modified version of the Honda-Tate theory applies to this category\forget{; compare Remark \ref{Rem_Honda-Tate_Theory} and Theorem \ref{Thm_Jannsen}}. We introduce the analog  $W_\ul\nu$ of the Weil pro-torus, and we discuss the Honda-Tate theory in Section~\ref{Sect_Quasi-isogeny classes and Honda-Tate theory}; see Propositions \ref{PropQIsoClassesI} and \ref{PropQIsoClassesII}, in which we obtain a \emph{quasi-isogeny criterion} for $C$-motives, and Theorem \ref{Thm_HT_C-Mot}. We further observe that

\begin{theorem}\label{Thm_HT_C-Mot}
There is a bijection 
$$
\text{the set $\Sigma$ of simple objects in $\mot(\ol{\BF}_q)$}\enspace \longleftrightarrow \enspace \text{$\Gamma_Q\backslash W_\ul\nu\times \BN_{\geqslant 1}\slash\!\!\sim$}.
$$
\end{theorem}

Note that the Honda-Tate theory of shtukas is also well studied in the PhD thesis of F.~R\"oting \cite{FelixThesis}. In the present article we give the proof for the fact that the map $\Sigma \rightarrow \Gamma_Q\backslash W_\ul\nu\times \BN_{\geqslant 1}\slash\!\!\sim$ is injective. This proof also appears in \cite{FelixThesis}, while the main achievement in \cite{FelixThesis} is the fact that this map is also surjective. We mention this in the proof of the above theorem.\\

\noindent
We also discuss the zeta-functions associated with $C$-motives and $\FG$-shtukas in section \ref{Section_Zeta-Function and analogous motives}. \\

The Semi-simplicity of the category of $C$-motives over $\ol\BF_q$ is proved in Theorem \ref{ThmPoincare-Weil}. This theorem is similar to the semisimplicity result of Jannsen~\cite{Jannsen} and states the following:

\begin{theorem}
The category $\mot(\ol\BF_q)$ with the fiber functor $\omega$ is a semi-simple tannakian category. In particular the kernel group $P:=\FP^{\Delta}$ of the corresponding motivic groupoid $\FP:=\Aut^\otimes\bigl(\omega|\mot(\ol\BF_q)\bigr)$ is a pro-reductive group over $Q_{\ol\BF_q}$. 
\end{theorem}

Notice that this theorem only holds for $L=\ol\BF_q$ and not for a finite field $L$. This is a remarkable difference to the number fields case. The reason behind this comes from the following elementary observation. Namely, unlike the characteristic zero case, a non semi-simple matrix may become semi-simple after raising to a relevant power. Apart from this difference, note in addition that the category of $C$-motives consists of ``mixed'' motives, in the sense that the eigenvalue of Frobenius endomorphism operating on the cohomology groups might have different absolute values;  compare also \cite{Goss}[Theorem 5.6.10]. In this sense, this result maybe also viewed as a (partial) analog for the semi-simplicity of the subcategory of mixed Tate motives over a finite field (with rational coefficient) inside a Voevodsky's motivic category. The latter fact can be derived from finiteness of $K$-groups over finite fields according to Quillen \cite{Qui}. \\

%Finally in the Appendix \ref{App.A}, we give an overview of the theory of motives over number fields and explain the analog with our theory over function fields. This should help interested readers to get a better view of the dictionary between number fields and function fields. In addition we try to explain how the category of $C$-motives may naturally arises  in this context. %In Appendix B we recall some details about the geometry of the (unbounded) moduli stack of global $\FG$-shtukas and the ind-algebraic structure on them, which was established in \cite{AH_Global}.

\section*{Acknowledgments}
 The authors would like to thank the unknown referee for careful reading and many useful comments and explanations which led to a considerable improvement of the exposition and simplified several proofs. \\
The first author is grateful to Chia-Fu Yu for his continues encouragements and limitless support. Moreover, he thanks Giuseppe Ancona, Arash Rastegar and Mehrdad Karimi for inspiring conversations and support.\\
\noindent
The authors acknowledge support of the DFG (German Research Foundation) in form of SFB 878 and Germany's Excellence Strategy EXC 2044--390685587 ``Mathematics M\"unster: Dynamics--Geometry--Structure''. The first author also acknowledges the support of NCTS (National Center For Theoretical Sciences) and IPM (Institute for Research in Fundamental Sciences).

\tableofcontents
\subsection{Notation and Conventions}\label{Notation and Conventions}

Throughout this article we denote by
\begin{tabbing}
$\genericG_\nu:=\FG\times_C\Spec Q_\nu,$\; \=\kill
$\BF_q$\> a finite field with $q$ elements and characteristic $p$,\\[1mm]
$\ol\BF_q$\> an algebraic closure of $\BF_q$,\\[1mm]
$L$\> a ring containing $\BF_q$,\\[1mm]

$\BaseOfD$\> a finite field containing $\BF_q$,\\[1mm]
%$\ol\BaseOfD$\> an algebraic closure of $\BF$,\\[1mm]

$C$\> \parbox[t]{0.711\textwidth}{a smooth projective geometrically irreducible curve over $\BF_q$,}\\[1mm]

$\eta$\> the generic point of $C$,\\[1mm]

$\nu$\> a closed point of $C$, also called a \emph{place} of $C$,\\[1mm]

$\ul\nu=(\nu_1,\ldots,\nu_n)$\> a tuple of $n$ pairwise different places $\nu_i$ on $C$,\\[1mm]

$\dot{C}:=C\setminus\ul\nu$\> the open subscheme of $C$,\\[1mm]

$Q:=\kappa(\eta)=\BF_q(C)$\> the function field of $C$,\\[1mm]

$\BF_\nu:=\kappa(\nu)$\> the residue field at the place $\nu$ on $C$,\\[1mm]

$A:=\Gamma(\dot{C},\CO_C)$\> the ring of regular functions outside $\ul\nu$,\\[1mm]

$A_L$\> the ring $A\otimes_{\BF_q}L$,\\[1mm]

$Q_L$\> the ring of fractions $\Frac(A_L)$,\\[1mm]

$A_\nu$\> the completion of the stalk $\CO_{C,\nu}$ at the place $\nu$,\\[1mm]
$Q_\nu:=\Quot(A_\nu)$\> its fraction field,\\[1mm]
$\hat{A}:=\BF\dbl z\dbr$\>  the ring of formal power series in $z$ with coefficients in $\BF$ ,\\[1mm]
$\wh Q:=\Frac(\hat{A})$\> its fraction field,\\[1mm]

$\BA^\ul\nu$\> the ring of integral adeles of $C$ outside $\ul\nu$,\\[1mm]

$\BA_Q^{\ul\nu}:=\BA^\ul\nu\otimes_{\CO_C} Q$\> the ring of adeles of $C$ outside $\ul\nu$,\\[1mm]

\forget
{
$\BD:=\Spec \BF\dbl z \dbr$ \> \parbox[t]{0.711\textwidth}{the spectrum of the ring of formal power series in $z$ with coefficients in $\BF$,}\\[1mm]
$\hat{\BD}:=\Spf \BF\dbl z \dbr$ \>\parbox[t]{0.711\textwidth}{ the formal spectrum of $\BF\dbl z\dbr$ with respect to the $z$-adic topology.}\\
}

$\BD_R:=\Spec R\dbl z \dbr$ \> \parbox[t]{0.711\textwidth}{the spectrum of the ring of formal power series in $z$ with coefficients in an $\BaseOfD$-algebra $R$,}\\[1mm]

$\hat{\BD}_R:=\Spf R\dbl z \dbr$ \>\parbox[t]{0.711\textwidth}{ the formal spectrum of $R\dbl z\dbr$ with respect to the $z$-adic topology.}\\

\end{tabbing}
\noindent
For a formal scheme $\wh S$ we denote by $\Nilp_{\wh S}$ the category of schemes over $\wh S$ on which an ideal of definition of $\wh S$ is locally nilpotent. We  equip $\Nilp_{\wh S}$ with the \'etale topology. We also denote by
\begin{tabbing}
$\genericG_\nu:=\FG\times_C\Spec \wh Q_\nu,$\; \=\kill
%$n\in\BN_{>0}$\> a positive integer,\\[1mm]
%$\ul \nu:=(\nu_i)_{i=1\ldots n}$\> an $n$-tuple of closed points of $C$,\\[1mm]
%$\BA_C^\ul\nu$\> the ring of rational adeles of $C$ outside $\ul\nu$,\\[1mm]
$A_\ul\nu$\> the completion of the local ring $\CO_{C^n,\ul\nu}$ of $C^n$ at the closed point $\ul\nu$,\\[1mm]

$\Nilp_{A_\ul\nu}:=\Nilp_{\Spf A_\ul\nu}$\> \parbox[t]{0.77\textwidth}{the category of schemes over $C^n$ on which the ideal defining the closed point $\ul\nu\in C^n$ is locally nilpotent,}\\[2mm]
$\Nilp_{\BaseOfD\dbl\zeta\dbr}:=\Nilp_{\hat\BD}$\> \parbox[t]{0.77\textwidth}{the category of $\BD$-schemes $S$ for which the image of $z$ in $\CO_S$ is locally nilpotent. We denote the image of $z$ by $\zeta$ since we need to distinguish it from $z\in\CO_\BD$.}\\[2mm]
$\FG$\> a flat affine group scheme of finite type over $C$,\\[1mm]
%$G$\> generic fiber of $\FG$,\\[1mm]
%$G$\> a reductive group over $\dot{\BD}$,\\[1mm]
$\BP_\nu:=\FG\times_C\Spec  A_\nu,$ \> the base change of $\FG$ to $\Spec A_\nu$,\\[1mm]
$\genericG_\nu:=\FG\times_C\Spec Q_\nu,$ \> the generic fiber of $\BP_\nu$ over $\Spec Q_\nu$,\\[1mm]

$\BP$\> a flat affine group scheme of finite type over $\BD=\Spec\BaseOfD\dbl z\dbr$,\\[1mm] 
$\genericG$\> the generic fiber of $\BP$ over $\Spec\BaseOfD\dpl z\dpr$.
\end{tabbing}

\noindent
Let $S$ be an $\BF_q$-scheme and consider an $n$-tuple $\ul s:=(s_i)_i\in C^n(S)$. We denote by $\Gamma_\ul s$ the union $\bigcup_i \Gamma_{s_i}$ of the graphs $\Gamma_{s_i}\subseteq C_S:=C\times_{\BF_q}S$. \\

\noindent
We denote by $\sigma_S \colon  S \to S$ the $\BF_q$-Frobenius endomorphism which acts as the identity on the points of $S$ and as the $q$-power map on the structure sheaf. We set
\begin{tabbing}
$\genericG_\nu:=\FG\times_C\Spec Q_\nu,$\; \=\kill
$C_S := C \times_{\Spec\BF_q} S$ ,\> and \\[1mm]
$\sigma := \id_C \times \sigma_S$.\\[1mm]

%Likewise we let $\hat{\sigma}_S:\CO_S\dbl z\dbr\to\CO_S\dbl z\dbr$ to be the morphism which is $\sigma_S$ on $\CO_S$ and maps $z$ to\\ itself. \urscomment{I think you want $\hat{\sigma}_S$ to be the $\#\BF$-Frobenius on $\CO_S$ !} We drop the subscript $S$ and simply write $\hat{\sigma}$, when it is precise from the context.\\

$Ch_{\sim, d}(-,\BQ)$\> the group of cycles of dimension $d$ modulo the equivalence relation $\sim$,\\ \> with coefficients in $\BQ$. Here $\sim$ stands for an adequate equivalence relation, \\ \> e.g. rational, algebraic, homological and numerical equivalence relations.
\end{tabbing}

\noindent
Let $H$ be a sheaf of groups (for the \'etale topology) on a scheme $X$. In this article a (\emph{right}) \emph{$H$-torsor} (also called an \emph{$H$-bundle}) on $X$ is a sheaf $\CG$ for the \fppf topology on $X$ together with a (right) action of the sheaf $H$ such that $\CG$ is isomorphic to $H$ on a \fppf covering of $X$. Here $H$ is viewed as an $H$-torsor by right multiplication. \\

\begin{definition}
\begin{enumerate}
\item
Let $R$ be a ring and $X\subseteq R$ be a subset. We denote by $C_R(X)$ the \emph{centralizer} of $X$ in $R$, i.e. 
$C_R(X):=\{a\in R; a\cdot x=x\cdot a \forall x\in X\}$. The subring $C_R(C_R(X))$ is called \emph{bicommutant} of $X$ in $R$. 
\item
Let $M$ be an $R$-module. Set $R_M:=\im \left(R\to \End (M,+), a\mapsto (a:m\mapsto a.m)\right)$
\item
We denote by $\CZ(R)$ the centralizer of $R$ in $R$, i.e. $\CZ(R):=C_R(R)$.
\item
We denote by $Bicom_R(M):=C_{\End(M,+)}(\End_R(M)) $
\end{enumerate}
\end{definition}

\begin{remark}\label{RemBicom}
\begin{enumerate}

\item\label{RemBicomA}
\cite[Chapitre 8, \S\,1, n$^{\rm o\!}$~2, Proposition 3]{BourbakiAlgebra}. Let $K$ be a field and $R$ and $R'$ be two $K$-algebras. Let $S\subseteq R$ and $S'\subseteq R'$ be sub-K-algebras. Then we have 
$$
C_{R\otimes_K R'}(S\otimes_K S')=C_R(S)\otimes_K C_{R'}(S') 
$$
in particular $\CZ(R\otimes_K R')=\CZ(R)\otimes_K \CZ(R')$. 
\item\label{RemBicomB}
\cite[Chapitre 8, \S\,4, n$^{\rm o\!}$~2, Corollaire 1]{BourbakiAlgebra}. Let $M$ be a semi-simple $R$-module which is finitely generated as an $\End_R(M)$-module. Then $Bicom_R(M)=R_M$.  

\end{enumerate}

\end{remark}

%\noindent

\section{Category of C-Motives And Realization Functors}\label{SectCategory of C-Motives And Realization Functors}
In this section we present the basic definitions of the category of $C$-motives.\\
Let $\ul\nu=(\nu_1,\ldots,\nu_n)$ be an $n$-tuple of closed points of $C$ and let $A:=\Gamma(\dot{C},\CO_C)$ be the coordinate ring of the open subscheme $\dot{C}:=C\setminus \{\nu_1,\ldots,\nu_n\}$. %Let $\BF_{\ul\nu}$ be the residue field of the point $\ul\nu\in C^n$.

\begin{definition}\label{DefCatCMotives}

\begin{enumerate}
\item
Let $S$ be a scheme over $\BF_q$. A \emph{$C$-motive $\ul\CM$} with characteristic $\ul\nu$ over $S$ is a tuple $(\CM,\tau_\CM)$ consisting of

\begin{enumerate}
\item \label{DefCatCMotives_i}
a locally free sheaf $\CM$ of $\CO_{C_S}$-modules of finite rank,

\item \label{DefCatCMotives_ii}
an isomorphism $\tau_\CM: \sigma^\ast \dot{\CM} \to \dot{\CM}$ where $\dot{\CM}$ denotes the pullback of $\CM$ under the inclusion $\dot{C}_S\to C_S$, and $\sigma=id \times \sigma_S$ where $\sigma_S:S\to S$ is the absolute Frobenius morphism over $\BF_q$.    
\end{enumerate}

\item 
A \emph{morphism} $\ul\CM\to\ul\CN$ is a homomorphism $f:\CM\to\CN$ of sheaves on $C_S$ which fits into the following commutative diagram 

$$
\CD
\sigma^\ast \dot{\CM} @>\tau_\CM|_{\dot{C}_S}>> \dot{\CM}\\
@V{\sigma^\ast f}VV @VVfV\\
\sigma^\ast \dot{\CN}@>\tau_\CN|_{\dot{C}_S}>> \dot{\CN}.
\endCD  
$$
The set of morphisms is denoted $\Hom_S(\ul\CM,\ul\CN)$.

The set of \emph{quasi-morphisms} $\QHom (\ul\CM,\ul\CN)$ consists of the equivalence classes of the commutative diagrams

$$
\CD
\sigma^\ast \dot{\CM} @>\tau_\CM>> \dot{\CM}\\
@V{\sigma^\ast f}VV @VVfV\\
\sigma^\ast \dot{\CN}\otimes\CO(D_S)@>\tau_\CN>> \dot{\CN}\otimes \CO(D_S),
\endCD  
$$

where $D$ is a divisor on $C$ and $D_S:=D\times_{\BF_q} S$, and two such diagrams for divisors $D$ and $D'$ are called equivalent provided that the corresponding diagrams agree when we tensor with $\CO(D_S+D_S')$.\\

Note that when $S=\Spec L$, for a field $L$, one can equivalently say that the set of quasi-morphisms $\QHom (\ul\CM,\ul\CN)$ is given by the following commutative diagrams 

$$
\CD
\sigma^\ast \CM_\eta @>\tau_{\CM,\eta}>> \CM_\eta\\
@V{\sigma^\ast f}VV @VVfV\\
\sigma^\ast \CN_\eta@>\tau_{\CN,\eta}>> \CN_\eta.
\endCD  
$$

Here $\CM_\eta$ denotes the pull back of $\CM$ under $\eta\times_{\BF_q} S \to C_S$.

\item

A \emph{quasi-isogeny}\forget{ (resp. An isogeny)} between $\ul \CM$ and $\ul \CN$ is a morphism in $\QHom (\ul\CM,\ul\CN)$\forget{ (resp. in $\Hom (\ul\CM,\ul\CN)$)} which admits an inverse. See also \cite[Theorem~5.12]{HartlIsog} and Proposition \ref{Prop_dual_isog} below.

\item
We denote by $\mot (S)$\forget{ (resp. $\mot (L)$)} the $Q$-linear\forget{ (resp. $A$-linear)} category whose objects are $C$-motives of characteristic $\ul\nu$ as above, with quasi-morphisms as its morphisms. We further denote by $\mot (S)^\circ$\forget{ (resp. $\mot (L)^\circ$)} the category obtained by restricting the set of morphisms to quasi-isogenies\forget{ (resp. isogenies)}. When $S=\Spec L$ we simply use the notation $\mot (L)$.
\end{enumerate}

\end{definition}

\begin{remark}\label{Rem_AlternativeDefForC-Motives}
Let $\BF_\ul\nu$ be the residue field of the point $\ul\nu\in C^n$. If $L$ is a field extension of $\BF_\ul\nu$, then we can modify the above definition by replacing \ref{DefCatCMotives_ii} as follows. Let $s_i: \Spec L \to C$ be
composition of the natural morphism $\Spec L \to \Spec \BF_\ul\nu \hookrightarrow C^n$
followed by the projection onto
the $i$-th component. It factors as $s_i
: \Spec L \to \Spec \BF_{\nu_i}
\hookrightarrow C$. We may replace \ref{DefCatCMotives_ii} by

(ii)' an isomorphism $\tau_\CM: \sigma^\ast \CM|_{C_L\setminus \cup_i \Gamma_{s_i}} \to \CM|_{C_L\setminus \cup_i\Gamma_{s_i}}$
outside the graphs $\Gamma_{s_i}$
of the morphisms $s_i: \Spec L \to C$, where $\sigma = \id \times \sigma_L$ and $\sigma_L : L \to L$ is the absolute Frobenius morphism over $\BF_q$.

We always have $\dot{C}_L \subseteq C_L\setminus \cup_i\Gamma_{s_i}$ but this inclusion is strict when $\BF_{\ul\nu} \neq \BF_q$.  So the difference
between conditions (ii) and (ii)' is that in (ii)’ it is required that $\tau_\CM$ is an isomorphism on the
larger set $C_L\setminus \cup_i \Gamma_{s_i}$. Condition (ii)' is more natural and better compatible with the theory of
global $\FG$-shtukas, see Section \ref{SectGSht} below. Note that one can choose an embedding $\BF_\ul\nu\to L$, when
$L$ is large enough, then the obvious functor from the category of $C$-motives over $L$ defined with $(ii)'$ to
the category of $C$-motives over $L$ defined with $(ii)$ is fully faithful with morphisms given by
$\QHom (\ul\CM,\ul\CN)$. 
%With morphisms given by $\QHom (\ul\CM,\ul\CN)$ it is an equivalence of categories (although with morphisms given by $\Hom_L(\ul\CM, \ul\CN)$ it is not).

\end{remark}

\begin{proposition}\label{PropMot'ToMot}
 The obvious functor $\mot(L)'\to \mot(L)$ from the category $\mot(L)'$ of $C$-motives over $L$ defined with $(ii)'$ to
the category of $C$-motives over $L$ defined with $(ii)$ is fully faithful. With morphisms given by
$\QHom_L (\ul\CM,\ul\CN)$ it is an equivalence of categories, but with morphisms given by $\Hom_L(\ul\CM, \ul\CN)$
it is not.

\end{proposition}

\begin{proof}
Full faithfulness follows from the definitions. Now let $\ul\CM\in\mot(L)$. We are going to define a $\ul\CM'\in\mot(L)'$ which is isomorphic to $\ul\CM$ in $\mot(L)$. For this purpose fix a point $\nu\in\ul\nu$ and let $A_\nu$ be the completion of the local ring $\CO_{C,\nu}$ and let $Q_\nu$ be the fraction field of $A_\nu$. Since the point $\nu$ splits in $C_L$ into $\deg(\nu)$-many $L$-valued points, the completed tensor products
\begin{eqnarray*}
A_\nu\wh{\otimes}_{\BF_q}L & = & \prod_{i\in\BZ/(\deg\nu)}(A_\nu\wh{\otimes}_{\BF_q}L)/\Fa_i\qquad\text{and} \\[2mm]
Q_\nu\wh{\otimes}_{\BF_q}L & = & \prod_{i\in\BZ/(\deg\nu)}(Q_\nu\wh{\otimes}_{\BF_q}L)/\Fa_i
\end{eqnarray*}
split correspondingly, where $\Fa_i= \langle \alpha\otimes 1-1\otimes \alpha^{q^i} : \alpha\in \BF_\nu\rangle$. The map $\sigma$ sends $\Fa_i$ to $\Fa_{i+1}$, and hence $\tau_\CM$ yields isomorphisms 
\[
\sigma^*\bigl(\dot\CM\otimes_{\CO_{\dot{C}_L}}(Q_\nu\wh{\otimes}_{\BF_q}L)/\Fa_i\bigr)\;=\;(\sigma^*\dot\CM)\otimes_{\CO_{\dot{C}_L}}(Q_\nu\wh{\otimes}_{\BF_q}L)/\Fa_{i+1}\;\isoto\;\dot\CM\otimes_{\CO_{\dot{C}_L}}(Q_\nu\wh{\otimes}_{\BF_q}L)/\Fa_{i+1}
\]
for all $i\in\BZ/(\deg\nu)$. Since $\dot{C}_L$ together with $\coprod_{\nu\in\ul\nu}\Spec A_\nu\wh{\otimes}_{\BF_q}L$ is an fpqc-covering of $C_L$ we may construct $\ul\CM'$ from $\ul\CM$ by modifying it in the components $\ul\CM\otimes_{\CO_{C_L}}(A_\nu\wh{\otimes}_{\BF_q}L)/\Fa_i$ for all $i\ne0$ at all $\nu\in\ul\nu$ according to
\[
\ul\CM'\otimes_{\CO_{C_L}}(A_\nu\wh{\otimes}_{\BF_q}L)/\Fa_i\;:=\;\tau_\CM^i\Bigl(\sigma^{i*}\bigl(\ul\CM\otimes_{\CO_{C_L}}(A_\nu\wh{\otimes}_{\BF_q}L)/\Fa_0\bigr)\Bigr)\;\subset\;\dot\CM\otimes_{\CO_{\dot{C}_L}}(Q_\nu\wh{\otimes}_{\BF_q}L)/\Fa_i\,.
\] 
Then by construction $\ul\CM':=(\CM',\tau_{\CM'})$ with $\tau_{\CM'}:=\tau_\CM$ satisfies
\[
\tau_{\CM'}: (\sigma^*\dot\CM)\otimes_{\CO_{\dot{C}_L}}(Q_\nu\wh{\otimes}_{\BF_q}L)/\Fa_i\;\isoto\;\dot\CM\otimes_{\CO_{\dot{C}_L}}(Q_\nu\wh{\otimes}_{\BF_q}L)/\Fa_i
\]
for all $i\ne0$, and so $\ul\CM'$ belongs to $\mot(L)'$. Since the restrictions of $\ul\CM$ and $\ul\CM'$ to $\dot{C}_L$ coincide, we have $\ul\CM\cong\ul\CM'$ in $\mot(L)$ as desired.
\end{proof}

\begin{proposition} Let $L$ be a field over $\BF_q$. The category $\mot(L)$ is a $Q$-linear rigid abelian tensor category. It further admits a fiber functor over the function field $Q_L$ of the curve $C_L$.
\end{proposition}

\begin{proof}

Let $f:\CM \to \CN(D)$ be a representative for a morphism in $\QHom(\ul\CM,\ul\CN)$. Then $\ul\ker f:=(\ker f, \tau_\CM|_{\sigma^\ast \ker f})$ and $\ul\im f:=(\im f, \tau_{\CN}|_{\sigma^\ast \im f})$. Consider the cokernel $\CF:=\coker f:\CM\to\CN(D)$ as a coherent sheaf of $\CO_{C_L}$-modules. The torsion subsheaf $\CT$ has finite support and $\CF/\CT$ is a torsion free sheaf. The morphism $\tau_\CN$ induces a morphism $\tau_\CF: \sigma^\ast \CF\to\CF$. We define $\ul\coker f:=(\CF/\CT,\tau_{\CF/\CT})$ and $\ul{\coim} f:=\ul\ker (\ul \CN\to \ul \coker f)$. Moreover this is clear that the natural morphism $\ul\im f\to\ul\coim f$ is a quasi-isogeny and therefore an isomorphism in $\mot(L)$\forget{ and thus the category is abelian}.\\ 
The tensor product of two $C$-motives $\ul \CM:=(\CM,\tau_\CM)$ and $\ul \CN:=(\CN,\tau_\CN)$ is the $C$-motive $\ul \CM \otimes \ul\CN$ consisting of the locally free sheaf of $\CO_{C_L}$-modules
$\CM \otimes_{\CO_{C_L}} \CN$ and the isomorphism $\tau_{\ul\CM\otimes\ul\CN}:=\tau_\CM\otimes \tau_\CN$. The unit object for the tensor product is $\ul{\mathds{1}}:=(\CO_{C_L},\id)$, and precisely we have $\QEnd(\ul{\mathds{1}})=Q$ . One can easily see that this category has an internal $\ul\Hom$ object. Namely we define $\ul\CH:=\ul\Hom (\ul \CM, \ul \CN)$ to be the object with $\CH:= \Hom_{\CO_{C_L}} (\CM, \CN)$ as the underlying locally free sheaf and $\tau_{\CH}$ is given by sending $h \in \CH$ to $\tau_\CN \circ h \circ \tau_\CM^{-1}$. This further defines the functor 
$$
\check{(-)}:=\ul\Hom(-,\mathds{1}): \mot(L) \to \mot(L),
$$ 
which sends $\ul\CM$ to its dual $\ul{\check{\CM}}$. Finally sending a $C$-motive $\ul\CM:=(\CM,\tau_\CM)$ to the generic fiber $\CM_\eta$ of the underlying locally free sheaf $\CM$ equips the category with a fiber functor 
$$
\omega(-):\mot(L)\to Q_L-\text{Vector spaces}.
$$

\end{proof}

\begin{proposition}\label{Prop_dual_isog}
Let $\ul\CM$ and $\ul\CN$ be $C$-motives, over a field $L$. Assume that $L$ is finite over $\BF_q$ of degree $s$. The following are equivalent
\begin{enumerate}
\item
$f\in \Hom(\ul\CM,\ul\CN)$ is a quasi-isogeny.
\item
There is a non-zero element $a\in A$ and $\check{f}\in \Hom(\ul\CN,\ul\CM)$ with $\check{f}\circ f=a. \id_{\CM}$ and $f\circ \check{f}=a. \id_{\CN}$.
\end{enumerate}
\end{proposition}

\begin{proof}

The proof is similar to \cite[Corollary 5.4]{Bor-Har2} and \cite[Proposition 3.1.2]{TaelmanII}. 
See also \cite[Corollary 5.15]{HartlIsog}. 

\end{proof}

\subsection{Crystalline and \'Etale realization functors and Tate conjecture}\label{SubSectRealizationFunctors}

In analogy with the theory of motives, the category of $C$-motives also admits crystalline and \'etale realizations. In this section we recall the definition of the corresponding realization categories and functors. We further prove the analog of the Tate conjecture for the category $C$-motives over a finite field. The results are similar to \cite[\S\S\,8, 9]{Bor-Har2}.

\subsubsection{Category Of (Iso-)Crystals}

Fix a place $\nu\in C$ and a uniformizer $z:=z_\nu\in A_\nu=\hat{\CO}_{C,\nu}$. It yields canonical isomorphisms $A_\nu=\BF_\nu\dbl z\dbr$ and $Q_\nu=\BF_\nu\dpl z\dpr$. Let $L/\BF_\nu$ be a field extension. Let $\hat{\sigma}_\nu$ be the endomorphism of $L\dbl z\dbr=A_\nu\wh{\otimes}{_{\BF_\nu}} L$ which is the $\#\BF_\nu$-Frobenius $b\mapsto b^{\#{\BF_\nu}}=b^{q^{\deg\nu}}$ on $L$ and fixes $z$. Also let $L\dpl z\dpr=Q_\nu\wh{\otimes}{_{\BF_\nu}} L$ denote the fraction field of $L\dbl z\dbr$.

\begin{definition}\label{Def_Crystal}
%Let $L/\BF_q$ be a field. Set $\hat{A}:=\BF\dbl z\dbr$ and $\wh{Q}:=\Frac(\hat{A})$.
Keep the above notation. 
\begin{enumerate}
\item
A $\hat{\sigma}_\nu$-\emph{crystal} $\hat{\ul M}$ (resp. $\hat{\sigma}_\nu$-\emph{iso-crystal}) of rank $r$ over $L$ is a tuple $(\hat{M},\hat{\tau})$  (resp. $(\dot{\hat{M}},\hat{\tau})$ ) consisting of the following data

\begin{enumerate}

\item a free $L\dbl z\dbr$-modulue $\hat{M}$ (resp. an $L\dpl z\dpr$-vector space $\dot{\hat{M}}$) of rank $r$,

\item an isomorphism $\hat{\tau}: \hat{\sigma}_\nu^\ast \hat{M}[1/z] \to \hat{M}[1/z]$ (resp. $\hat{\tau}: \hat{\sigma}_\nu^\ast \dot{\hat{M}} \to \dot{\hat{M}}$), where $\hat{M}[1/z]:=\hat{M}\otimes_{L\dbl z \dbr} L\dpl z \dpr$.

\end{enumerate}
A $\hat{\sigma}_\nu$-crystal is \emph{\'etale} if the isomorphism $\hat{\tau}: \hat{\sigma}_\nu^\ast \hat{M}[1/z] \to \hat{M}[1/z]$ comes from an isomorphism $\hat{\tau}: \hat{\sigma}_\nu^\ast \hat{M} \to \hat{M}$.\\ 

\item A \emph{quasi-morphism} (resp. \emph{morphism}) between $\hat{\sigma}_\nu$-crystals $\hat{\ul M}:=(\hat{M},\hat{\tau})$ and $\hat{\ul M}':=(\hat{M}',\hat{\tau}')$ is a morphism $f:\hat{M}[1/z]\to \hat{M}'[1/z]$  (resp. $f:\hat{M}\to \hat{M}'$) such that $f\circ \hat{\tau}=\hat{\tau}' \circ \hat{\sigma}_\nu^\ast f$. We denote by $\hat{\sigma}_\nu$-$\textbf{CrystIso}(L)$ (resp. $\hat{\sigma}_\nu$-$\textbf{Cryst}(L)$) the $Q_\nu$-linear (resp. $A_\nu$-linear) category of $\hat{\sigma}_\nu$-crystals together with quasi-morphisms (resp. morphisms) as its morphisms. 
We denote by $\textbf{\'Et~$\hat{\sigma}$-CrystIso}(L)$ (resp. $\textbf{\'Et~$\hat{\sigma}$-Cryst}(L)$) the full $Q_\nu$-linear (resp. $A_\nu$-linear) subcategory of $\hat{\sigma}$-$\textbf{CrystIso}(L)$ (resp. $\hat{\sigma}$-$\textbf{Cryst}(L)$) consisting of \'etale $\hat{\sigma}$-crystals.

\item
A \emph{quasi-morphism} between $\hat{\sigma}_\nu$-iso-crystals $(\dot{\hat{M}},\hat{\tau})$ and $(\dot{\hat{M}}',\hat{\tau}')$ is a morphism $f:\dot{\hat{M}}\to \dot{\hat{M}}'$ such that $f\circ \hat{\tau}=\hat{\tau}' \circ \hat{\sigma}_\nu^\ast f$. We denote by $\hat{\sigma}_\nu$-$\textbf{IsoCryst}(L)$ the $Q_\nu$-linear category of $\hat{\sigma}_\nu$-iso-crystals together with quasi-morphisms as its morphisms.

\end{enumerate}

\end{definition}

\begin{definition}\label{Def_EtRealization}
The \emph{first \'etale cohomology realization} of an \'etale $\hat{\sigma}_\nu$-crystal $\hat{\ul M}$ is the $\Gamma_L := Gal(L^{sep}/L)$-module of $\hat{\tau}$-invariants
$$
\Koh_{\text{\'et}}^1(\hat{\ul M},A_\nu) \;:=\; \left(\hat{M} \otimes_{L\dbl z\dbr} L^\sep\dbl z\dbr \right)^{\hat{\tau}}\;:=\;\bigl\{\,m\in\hat{M} \otimes_{L\dbl z\dbr} L^\sep\dbl z\dbr\colon m=\hat\tau(\hat\sigma^*_\nu m)\,\bigr\}.$$ 
We set $ \Koh_{\text{\'et}}^i(\hat{\ul M},A_\nu) := \wedge^i \Koh_{\text{\'et}}^1(\hat{\ul M},A_\nu)$ and $\Koh_{\text{\'et}}^i(\hat{\ul M},Q_\nu):=\Koh_{\text{\'et}}^i(\hat{\ul M},A_\nu)\otimes_{A_\nu} Q_\nu.
$
 
\end{definition}

\noindent
We recall the following crucial result.
 
\begin{theorem}\label{Thm_EmbeddingOfEtCrystintoGalMods}
We have the following statements for an \'etale $\hat\sigma_\nu$-crystal $\hat{\ul M}$.

\begin{enumerate}
\item
$\Koh_{\text{\'et}}^1(\hat{\ul M},A_\nu)$ is a free $A_\nu$-module of rank equal to $\rk\hat{\ul{M}}$ and the following natural morphism
 $$
\Koh_{\text{\'et}}^1(\hat{M},A_\nu) \otimes_{A_\nu} L^\sep\dbl z\dbr\to \hat{M}\otimes_{L\dbl z\dbr}L^\sep\dbl z\dbr
 $$ 
is a $\Gamma_L\times\hat{\tau}$-equivariant isomorphism of $L^\sep\dbl z\dbr$-modules, where on the left hand side $\Gamma_L$-acts diagonally and $\hat{\tau}$ acts as $\id\otimes\hat{\sigma}_\nu^\ast$ and on the right hand side $\Gamma_L$ only acts on $L^\sep\dbl z\dbr$ and $\hat\tau$ acts as $m\otimes f\mapsto\hat\tau(\hat\sigma^*_\nu m)\otimes\hat\sigma^*_\nu(f)$. In particular 
$\Koh_{\text{\'et}}^1(\hat{\ul M},A_\nu)$ is a free $A_\nu$-module of rank $\rk \hat{\ul M}$. 

\item
The first \'etale cohomology functor 
$$
\Koh_{\text{\'et}}^1(-,A_\nu):\textbf{\'Et~$\hat{\sigma}_\nu$-Cryst}(L)\to A_\nu[\Gamma_L]\text{-modules}
$$ 
is a fully faithful embedding of the category of $\textbf{\'Et~$\hat{\sigma}_\nu$-Cryst}(L)$ into the category of $A_\nu[\Gamma_L]$-modules. One can recover $\hat{\ul M}$ from its \'etale realization by taking Galois invariants

$$
\ul{\hat{M}}:=\left( \Koh_{\text{\'et}}^1(\hat{\ul M},A_\nu)\otimes_{A_\nu} L^\sep\dbl z\dbr\right)^{\Gamma_L}.
$$

\end{enumerate}

\end{theorem}

\begin{proof}
See \cite[Proposition 6.1]{TW} and \cite[Proposition 1.3.7 and A.4.5]{HartlPSp}. Compare also to \cite[Proposition~8.4]{Bor-Har2}.
\\
\end{proof}

\begin{proposition}\label{Prop_RationalTateisexactI}

The first \'etale cohomology functor 

$$
\Koh_{\text{\'et}}^1(-, A_\nu): \textbf{\'Et~$\hat{\sigma}_\nu$-Cryst}(L)\to A_\nu[\Gamma_L]-\text{modules}
$$
is exact.

%(resp. $\Koh_{\text{\'et}}^1(-, Q_\nu): \textbf{\'Et~$\hat{\sigma}_\nu$-CrystIso}(L)\to Q_\nu[\Gamma_L]-\text{modules}$) is exact (resp. exact). 

\end{proposition}
 
\begin{proof}

Given a short exact sequence 

$$
0 \to \ul{\hat{M}}''\to \hat{\ul M}\to \hat{\ul M}'\to 0 
$$
in \textbf{\'Et~$\hat{\sigma}_\nu$-Cryst}(L), we obtain the following exact sequence
$$
0 \to \Koh_{\text{\'et}}^1(\ul{\hat{M}}'',A_\nu)\otimes_{A_\nu}L^\sep\dbl z\dbr \to \Koh_{\text{\'et}}^1(\hat{\ul M},A_\nu)\otimes_{A_\nu}L^\sep\dbl z\dbr\to \Koh_{\text{\'et}}^1(\hat{\ul M}',A_\nu)\otimes_{A_\nu}L^\sep\dbl z\dbr \to 0, 
$$
according to Theorem \ref{Thm_EmbeddingOfEtCrystintoGalMods} a) and faithfully flatness of $L\dbl z \dbr \to L^\sep\dbl z \dbr$. By faithfully flatness of $A_\nu\to A_{\nu,L^\sep}:=L^\sep\dbl z \dbr$, the above short exact sequence descends and yields the following short exact sequence
$$
0 \to \Koh_{\text{\'et}}^1(\ul{\hat{M}}'',A_\nu)\to \Koh_{\text{\'et}}^1(\hat{\ul M},A_\nu)\to \Koh_{\text{\'et}}^1(\hat{\ul M}',A_\nu)\to 0.
$$

\end{proof}

%\subsubsection{isocrystals and slope decomposition} 

%\comment{explain this}

\bigskip

\noindent
\subsubsection{Crystalline realization of $C$-Motives}\label{SubSecCrysRealization}

Let $\nu\in C$ be a place, let $L/\BF_q$ be a field extension, and consider the ring $A_{\nu,L}:=A_\nu\wh\otimes_{\BF_q}L=(\BF_\nu\otimes_{\BF_q}L)\dbl z\dbr$. Let $\ell=\{ \alpha\in L; \alpha^{\# \BF_\nu}=\alpha\}$ be the intersection of $\BF_\nu$ and $L$, i.e. those elements $\alpha\in L$ such that $\alpha^{q^{\deg \nu}}=\alpha$. Let $s$ denote the degree of the field extension $\ell/\BF_q$. Then $\#\ell=q^s$. The scheme $\Spec (A_{\nu,L})$ is the union of its connected components. We write $\Spec A_{\nu,L}=\amalg_{i\in\BZ/s} V(\Fa_{\tilde{\nu_i}})$. The connected component $V(\Fa_{\tilde{\nu_i}}):=\Spec (A_{\nu,L})/\Fa_{\tilde\nu_i}=\Spec A_\nu\wh{\otimes}_\ell L$ corresponds to the ideal $\Fa_{\tilde\nu_i}= ( a\otimes 1-1\otimes a^{q^i} : a \in \ell)$ and $\sigma$ cyclically permutes these components and $\sigma^s$ leaves each of the components $V(\Fa_{\tilde\nu_i})$ stable. Thus the set of connected components is endowed with a free $\BZ/s:=\Gal(\ell/\BF_q)$-action. Here $\tilde\nu_i$ stands for the places of $C_\ell$ lying above $\nu\in C$. Note that when $L$ contains $\BF_\nu$, then $\ell=\BF_\nu$ and $s=\deg\nu$.

\begin{definition}\label{Def-RemCryst'}

In the above situation define the category $\textbf{\'Et~$(\sigma,\nu)$-CrystIso}(L)$ as the category whose objects consist of tuples $(\hat{M},\hat{\tau})$, where $\hat{M}$ is a locally free $A_{\nu,L}$-module and $\tau: \sigma^{\ast}\hat{M}\to \hat{M}$ is an isomorphism. One defines quasi-morphisms in a similar way as in definition \ref{Def_Crystal} and considers them as the set of morphisms of this category. Similarly one may define the categories $\textbf{\'Et~$(\sigma,\nu)$-Cryst}(L)$ and $\textbf{$(\sigma,\nu)$-IsoCryst}(L)$. 

\end{definition}

\begin{proposition}\label{Prop_RedisIsom}

Let $L/\BF_\nu$ be a field extension, and let $\tilde\nu_i$ be as above. The reduction modulo $\Fa_{\tilde\nu_i}$ induces equivalences of categories, which are independent of $i\in\BZ/\deg\nu$
\begin{alignat*}{4}
Red: &\quad\textbf{\'Et~$(\sigma,\nu)$-Cryst}(L)~ && \tilde{\longrightarrow}~ && \textbf{\'Et~$\hat\sigma_\nu$-Cryst}(L) \\
Red: & \quad\textbf{\'Et~$(\sigma,\nu)$-CrystIso}(L)~ && \tilde{\longrightarrow} && \textbf{\'Et~$\hat\sigma_\nu$-CrystIso}(L) \\
Red: & \quad\textbf{$(\sigma,\nu)$-IsoCryst}(L)~ && \tilde{\longrightarrow} && \textbf{$\hat\sigma_\nu$-IsoCryst}(L)\\
& \qquad\ul{\hat{M}}:=(\hat{M},\hat{\tau}) && \!\longmapsto\, &&  Red~\ul{\hat{M}}:=(\hat{M}\slash \Fa_{\tilde\nu_i},\hat{\tau}^{\deg\nu}) 
\end{alignat*}
%In addition, on the category of \'etale $\sigma$-crystals the above functor preserves the \'etale realization, i.e. $\Koh_{\text{\'et}}^1(\ul{\hat{M}},A_\nu)=\Koh_{\text{\'et}}^1(Red~\ul{\hat{M}},\hat{A}_{\tilde\nu_i})$ and $\Koh_{\text{\'et}}^1(\ul{\hat{M}},A_\nu)=\Koh_{\text{\'et}}^1(Red~\ul{\hat{M}},\hat{A}_{\tilde\nu_i})$.

\end{proposition}

\begin{proof}
This is \cite[Proposition~8.5]{Bor-Har2}.
\end{proof}

\begin{definition}\label{Def_Gamma_nuI}
Let $\nu\in C$ be a place and let $L/\BF_\nu$ be a field extension. For $\ul\CM:=(\CM,\tau_\CM)$ in $\mot(L)$, let $\hat{M}$ denote the pull back $\CM\otimes_{\CO_{C_L}} A_{\nu,L}$. Sending $\ul\CM$ to $(\hat{M},\tau_\CM\otimes 1)$ and further to $(\hat{M}/\Fa_{\tilde\nu_i},(\tau_\CM\otimes 1)^{\deg\nu})$ defines a functor which is independent of the choice of $i\in\BZ/\deg\nu$

$$
\Gamma_\nu (-): \mot(L)\longto\textbf{$(\sigma,\nu)$-IsoCryst}(L)\longto\textbf{$\hat\sigma_\nu$-IsoCryst}(L).
$$
We call this the \emph{crystalline realization functor at $\nu$}.
If $\nu$ does not lie in the characteristic places $\ul\nu$ the functor produces \'etale crystals
$$
\Gamma_\nu (-): \mot(L)\longto\textbf{\'Et~$(\sigma,\nu)$-Cryst}(L)\longto\textbf{\'Et~$\hat\sigma_\nu$-Cryst}(L).
$$
\end{definition}

\begin{remark}
The above construction at $\nu\in\ul\nu$ is apparently more canonical when we work with the alternative definition of $\mot(L)$, explained in Remark \ref{Rem_AlternativeDefForC-Motives}. As we will see in Section \ref{SectGSht}, this is in fact the case when we work with ($G$-)$C$-motives arising from ($\FG$-)shtukas. This is because there is a preferred canonical section $\Spec L\to \Spec \BF_\ul\nu\to C^n$ given by the characteristic sections (also called legs) of the ($\FG$-)shtuka. 

\end{remark}

\subsubsection{\'Etale realization of $C$-Motives and Tate Conjecture}

When $\nu$ does not lie in the characteristic places $\ul\nu$, assigning the $\Gamma_L$-module $\Koh_{\text{\'et}}^i(\hat{\ul M}, A_\nu)$ to the \'etale crystal $\hat{\ul M}:=\Gamma_\nu(\ul\CM)$, defines a functor

$$
\Koh_{\text{\'et}}^i(-,A_\nu): \mot(L)\;\longto\; \textbf{$A_\nu[\Gamma_L]$-modules},\quad \ul\CM\;\longmapsto\;\Koh_{\text{\'et}}^i(\Gamma_\nu(\ul\CM),A_\nu).
$$
\noindent
Tensoring up with $Q_\nu$ we similarly define

$$
\Koh_{\text{\'et}}^i(-,Q_\nu): \mot(L)\;\longto\; \textbf{$Q_\nu[\Gamma_L]$-modules},\quad \ul\CM\;\longmapsto\;\Koh_{\text{\'et}}^i(\Gamma_\nu(\ul\CM),A_\nu)\otimes_{A_\nu}Q_\nu.
$$

We call the above functor the i'th \'etale realization functor with coefficients in $A_\nu$ (resp. $Q_\nu$). We also use the notation $\omega^\nu(-)$ (resp. $\omega_{Q_\nu}^\nu(-)$) for the first cohomology functor $\Koh_{\text{\'et}}^1(-,A_\nu)$ (resp. $\Koh_{\text{\'et}}^1(-,Q_\nu)$).

 \begin{proposition}\label{Prop_RationalTateisexactII}

The functor $\omega_{Q_\nu}^\nu(-)$ is exact. 
 
 \end{proposition}
 
 \begin{proof}
This follows from exactness of $\Gamma_\nu(-)$ and Proposition~\ref{Prop_RationalTateisexactI}. 
 \end{proof}

\begin{proposition}\label{QHomhasfiniterank}
Let $\ul\CM$ and $\ul\CM'$ be in $\mot(L)$, and let $\nu$ be a place on $C$ different from all the characteristic places $\nu_i$. Then the obvious morphism 
$$
\QHom(\ul\CM,\ul\CM')\otimes_Q Q_\nu\;\longto\; \Hom_{Q_\nu[\Gamma_L](\omega_{Q_\nu}^\nu(\ul\CM),\omega_{Q_\nu}^\nu(\ul\CM'))}
$$ 
is injective. In particular $\dim_Q \QHom(\ul\CM,\ul\CM')\leq \rk \ul\CM \cdot \rk \ul\CM'$. 
\end{proposition}

\begin{proof}
Clearly we have an embedding of $\QHom(\ul\CM,\ul\CM')\otimes_Q Q_\nu$ in $\Hom_{Q_L}(\CM_\eta,\CM_\eta')\otimes_Q Q_\nu$. The latter sits inside $\Hom_{Q_{\nu,L}}(\Gamma_\nu(\ul\CM),\Gamma_\nu(\ul\CM'))$ and is compatible with respect to $\tau$ and $\tau'$, as well as $\Gamma_\nu(\tau)$ and $\Gamma_\nu(\tau')$. Here $Q_{\nu,L}$ denotes the total ring of fractions of $A_{\nu,L}$. Therefore we have an embedding 
$$
\QHom(\ul\CM,\ul\CM')\otimes_Q Q_\nu \hookrightarrow \Hom_{\textbf{Cris}}(\Gamma_\nu(\ul\CM),\Gamma_\nu(\ul\CM')) \otimes_{A_\nu} Q_\nu
$$
\noindent
The latter equals $\Hom_{Q_\nu(\Gamma_L)}(\omega_{Q_\nu}^\nu(\ul\CM),\omega_{Q_\nu}^\nu(\ul\CM'))$, see Theorem~\ref{Thm_EmbeddingOfEtCrystintoGalMods}.
\end{proof}

\noindent
The following Theorem can be regarded as an analog for Tate conjecture.

\begin{theorem}\label{Thm_Tate_Conj}
Let $\ul\CM$ and $\ul\CM'$ be $C$-motives over a finite field $L$. Let $\nu$ be a closed point of $C$ away from the characteristic places $\nu_i$. Then there are isomorphisms

$$
\Hom_L(\ul\CM,\ul\CM')\otimes_A A_\nu \tilde{\longrightarrow} \Hom_{{A_\nu}[\Gamma_L]}(\omega^\nu(\ul\CM),\omega^\nu(\ul\CM')) 
$$

(resp.
$$
\QHom_L(\ul\CM,\ul\CM')\otimes_Q Q_\nu \tilde{\longrightarrow} \Hom_{{Q_\nu}[\Gamma_L]}(\omega_{Q_\nu}^\nu(\ul\CM),\omega_{Q_\nu}^\nu(\ul\CM')) ) 
$$

of $A_\nu$-modules (resp. $Q_\nu$-vector spaces). Moreover if $\ul\CM=\ul\CM'$ then the above are ring isomorphisms. 
\end{theorem}

\begin{proof}

Without loss of generality we may assume that $\BF_\nu=\BF_q$. For this first observe that we can replace $\BF_q$ by $\ell$ (see subsection \ref{SubSecCrysRealization}), in order to assume that the coefficient ring $(\BF_\nu \otimes_\ell L)\dbl z\dbr$ of local shtukas is the power series ring (and not the product of power series rings); see Proposition \ref{Prop_RedisIsom}. Then one can proceed by replacing $\BF_\nu \otimes_\ell L$ by $L$. Now we claim that there is an isomorphism
$$
\Hom_L(\ul\CM,\ul\CM')\otimes_A A_\nu \tilde{\longrightarrow} \Hom_{L}(\Gamma_\nu(\ul\CM),\Gamma_\nu(\ul\CM')) 
$$
To see this, consider the following exact sequence
\[
\xymatrix @R=0.2pc {
0 \ar[r] & \Hom_L(\ul\CM,\ul\CM') \ar[r] & \Hom_{A_L} (\CM,\CM') \ar[r] & \Hom_A\left(\CM,\CM'\right) &\\
& & f \ar@{|->}[r] & f\circ \tau_\CM-\tau_{\CM'}\circ f. & \\
}
\]
Let us set $\ul{\hat{M}}:=\Gamma_\nu (\ul\CM)$ and $\ul{\hat{M}}':=\Gamma_\nu(\ul\CM')$. As $A\to A_\nu$ is flat, we get the following exact sequence

\[
\xymatrix @R=0.2pc {
0\ar[r] & \Hom_L(\ul\CM,\ul\CM')\otimes_A A_\nu \ar[r] & \Hom_{A_{\nu,L}} (\hat{M},\hat{M}') \ar[r] & \Hom_{A_\nu}(\hat{M},\hat{M}') &\\
& &  f \ar@{|->}[r] & f\circ \hat{\tau}_{\hat{M}}-\hat{\tau}_{\hat{M}'}\circ f, &
}
\]

\noindent
and thus $
\Hom_L(\ul\CM,\ul\CM')\otimes_A A_\nu \cong \Hom_L(\ul{\hat{M}},\ul{\hat{M}}') $. Note that we here use the fact that $A_L\otimes_A A_\nu = A_{\nu,L}$. This is only true when $L$ is a finite field, because $A_{\nu,L}$ is $\nu$-adically complete and the left hand side is not in general. It remains to show that

$$
\Hom_L(\ul{\hat{M}},\ul{\hat{M}}')=\Hom_{A_v[\Gamma_L]}(\omega^\nu(\ul\CM),\omega^\nu(\ul\CM')).
$$
But this follows from Theorem \ref{Thm_EmbeddingOfEtCrystintoGalMods} b).

\end{proof}

\section{The Endomorphism Ring Over Finite Fields}

Throughout this section we assume that the field $L$ is a finite field extension of $\BF_q$.

\begin{definition}\label{Def_FrobEndom}
Let $L/\BF_q$ be a field extension of degree $e$ and let $\ul\CM:=(\CM,\tau)$ be a $C$-motive in $\CM ot_C^\ul\nu(L)$. The \emph{Frobenius quasi-isogeny} $\pi:=\pi_\ul\CM\in \QEnd(\ul\CM)$ is given by
$$
\pi:=\tau \circ (\sigma^\ast) \tau \dots \circ (\sigma^\ast)^{e-1} \tau : \ul\CM=(\sigma^\ast)^e\ul\CM\to\ul\CM
$$
\end{definition}

\begin{proposition} 
Let $\ul \CM:=(\CM,\tau)$ be a $C$-motive over a finite field $L=\BF_{q^e}$ and let $\nu$ be a place on $C$ distinct from characteristic places $\nu_i$. Then
\begin{enumerate}

\item
The action of the topological generator $\Frob_L\colon x\mapsto x^{\#L}$ of $\Gamma_L$ on $\omega_{Q_\nu}^\nu(\ul\CM)$ equals the action of $\pi_\nu^{-1}$. Here $\pi_\nu$ denotes $\omega_{Q_\nu}^\nu(\pi)$.

\item
The image of the continuous morphism $A_\nu[\Gamma_L] \to \End_{A_\nu}(\omega_{Q_\nu}^\nu(\ul\CM))$ equals $A_\nu[\pi_\nu]$.

\end{enumerate}

\end{proposition}

\begin{proof}
\bigskip

%PROOF OF a) AVOIDING Lang's theorem:\\

a) Let $(\hat{M},\hat{\tau})$ be the crystal associated with $\ul\CM$ under $\Gamma_\nu(-)$. Recall that $\hat{\ul M}$ can be recovered from it's \'etale realization $\Koh_{\text{\'et}}^1(\hat{\ul M},A_\nu)$:

$$
\hat{M} \cong \left(\Koh_{\text{\'et}}^1(\hat{\ul M},A_v)\otimes_{A_v}L^\sep\dbl z\dbr\right)^{\Gamma_L},
$$
see Theorem \ref{Thm_EmbeddingOfEtCrystintoGalMods}, and $\hat{\tau}$ can be recovered as $1\otimes \hat{\sigma}$. Since $\phi_L=\hat{\sigma}^e$ and $\Gamma_L$ acts diagonally on $\Koh_{\text{\'et}}^1(\hat{\ul M},A_\nu)$, we have $\hat{\sigma}^e\otimes\hat{\sigma}^e = \id$ on $\hat{M}$ and thus we see that $\pi_\nu=\phi_L^{-1}$. \\

\bigskip
\forget{
\noindent
PROOF OF a) USING Lang's theorem:\\
a) One can reduce to the case $\BF_\nu=\BF_q$. Let $(\hat{M},\hat{\tau})$ be the crystal associated with $\ul\CM$ under $\Gamma_\nu(-)$. Fix an isomorphism $(\ul{\hat{M}},\tau)\tilde{=}(L \dbl z \dbr^r, T\cdot\hat{\sigma})$.
\forget{
\noindent
Then we have

$$
\pi_\nu=T\sigma(T)\cdots \sigma^{e-1}\cdot T\sigma^e.
$$
}
\noindent
Take $\Phi\in\GL_r(L^\sep\dbl z\dbr)$, invariant under $T\cdot\hat{\sigma}$. We may write

$$
\pi_\nu=\Phi^{-1}T\hat{\sigma}(T)\cdots\hat{\sigma}^{e-1}(T)\Phi=\hat{\sigma}^e(\Phi)^{-1}\cdot \Phi=\left(\Phi\cdot\hat{\sigma}^e(\Phi)\right)^{-1}=\rho_{\ul\CM,\nu}(\phi_L)^{-1}\\
$$ 
------------
}

%PROOF OF b) AVOIDING Lang's theorem:\\

b) Follows formally from the continuity of the action of $\Gamma$ on $\omega^\nu(\ul\CM)$. Namely, since the target of $\Gamma\to\End_{A_\nu}(\omega^\nu(\ul\CM))$ is finitely generated over $A_\nu$, the $A_\nu$-submodule $A_\nu[\pi_\nu]$ is closed. Since $\phi_L^{-1}$ generates a dense subgroup of $\Gamma_L$, the image of $\Gamma_L$ is contained in the closed submodule $A_\nu[\pi_\nu]$. 

\forget{
PROOF OF b) USING Lang's theorem:\\

b) Write $\Phi\equiv \sum_{i=0}^{n-1}\Phi_i z^i (\mod z^n)$ and let $K=L(\Phi_i; 0\leq i\leq n-1)$. For $\phi\in \Gamma_L$, observe that 
$$
\rho_{\ul\CM,\nu}(\phi)\equiv \left(\sum_{i=0}^{n-1}\Phi_i z^i\right)^{-1}\phi \left(\sum_{i=0}^{n-1}\Phi_i z^i\right) ~(\mod z^n)
$$ 
only depends on the image of $\phi$ under the natural morphism $\Gamma_L\to \Gal(K/L)=\{\phi_L^j;~0\leq j< [K:L]\}$. Therefore $\rho_{\ul\CM,\nu}(\phi)~(\mod z^n)$ lies in $\{\rho_{\ul\CM,\nu}(\phi_L^j);~0\leq j< [K:L]\}=A_\nu[\pi_\nu]/z^nA_\nu[\pi_\nu]$.

As $A_\nu[\rho_{\ul\CM,\nu}(\Gamma_L)]$ and $A_\nu[\pi_\nu]$ coincide modulo $z^n$ and the rings are $z$-adic complete, the statement follows. }

\end{proof}

Fix a $C$-motive $\ul\CM$ over $L$ and a place $\nu$ different from characteristic places $\nu_i$. Consider $F:=Q[\pi]\subseteq E:=\QEnd(\ul\CM)$. Note that $F$ is finitely generated as a $Q$-module by Proposition \ref{QHomhasfiniterank} and write $F:=Q[x]/\mu_\pi$, where $\mu_\pi$ is the minimal polynomial of $\pi$ over $Q$. Note further that according to Theorem \ref{Thm_Tate_Conj} we have $E\otimes_Q Q_\nu\cong E_\nu:=\End_{Q_\nu[\Gamma_L]} (\omega_{Q_\nu}^\nu(\ul\CM))$ and observe that under this isomorphism $F\otimes_Q Q_\nu$ gets identified with image $F_\nu$ of $Q_\nu[\Gamma_L]$ in $E_\nu$. Let $\chi_\nu$ denote the characteristic polynomial corresponding to $\pi_\nu:=\pi\otimes 1$  in $E_\nu$.

\begin{remark}\label{Rem_HomDec}
 Let $\ul\CM$ and $\ul\CM'$ be in $\mot(L)$ and let $V_\nu:=\omega_{Q_\nu}^\nu(\ul\CM)$ and $V_\nu':=\omega_{Q_\nu}^\nu(\ul\CM')$ at a place $\nu\in C$, different from characteristic places $\nu_i$. Furthermore, assume that the corresponding Frobenius endomprphisms $\pi_\nu\in \End_{Q_\nu[\Gamma_L]}(V_\nu)$ and $\pi_\nu'\in \End_{Q_\nu[\Gamma_L]}(V_\nu')$ are semisimple. Consider the decomposition $\chi_\nu=\prod_\mu \mu^{m_\mu}$ (resp. $\chi_\nu'=\prod_\mu \mu^{m_\mu'}$) of the characteristic polynomial $\chi_\nu$ (resp. $\chi_\nu'$) to its irreducible factors and set $K_\mu:=Q_\nu[x]/\mu$. There are   decompositions $V_\nu \cong \bigoplus_\mu (K_\mu)^{m_\mu}$ and $V_\nu' \cong \bigoplus_\mu (K_\mu)^{m_\mu'}$, and therefore

$$
\Hom_{Q_\nu[\Gamma_L]}(V_\nu,V_\nu')\cong\bigoplus_i Mat_{m_\mu'\times m_\mu} (K_\mu).
$$

\end{remark}

\begin{definition}
A $C$-motive $\ul\CM\neq 0$ in $\mot(L)$ is called \emph{simple}, if it has no non-trivial quotient in $\mot(L)$. Furthermore $\ul\CM$ is called \emph{semi-simple}, if it admits a decomposition into a direct sum of simple ones.\\
\end{definition}

The following theorem generalizes Proposition 6.8 and Theorem 6.11 of \cite{Bor-Har1} and as in \cite{Jannsen}, is a key observation for the proof of the semi-simplicity of the category $\mot(\ol\BF_q)$. The proof given bellow is a slight modification of the proof given in \cite{Bor-Har1}.

\begin{theorem}\label{Thm_SemisimplicityQEnd}

Consider the following diagram

\[
\xymatrix {
\mot(L)^\circ \ar@/^-2pc/[drr]_{\End_{Q_\nu[\Gamma_L]}(\omega_{Q_\nu}^\nu(-))} \ar[rr]^{{\QEnd(-)}} & & Q\text{-Algebras}\ar[d]^{-\otimes Q_\nu} & \\
 &  &Q_\nu\text{-Algebras}. \\
&   \\
}
\]

The following statements are equivalent

\begin{enumerate}

\item
$\ul\CM\in\mot(L)$ is semi-simple. 

\item
$E:=\QEnd(\ul\CM)$ is semi-simple. 

\item

$\End_{Q_\nu[\Gamma_L]}(\omega_{Q_\nu}^\nu(\ul\CM))$ is semisimple. 

\item

$F_\nu:=Q_\nu[\pi_\nu]$ is semisimple. 

\item

$F$ is semi-simple.

%If $\ul\CM\in\mot(L)$ is semi-simple then the above functor sends the full subcategory of semisimple motives to semisimple $Q$-algebras. Furthermore $E$ is semisimple if and only if $E_\nu$ is semisimple.

\end{enumerate}

\end{theorem}

\begin{proof}
(a) $\Rightarrow$ (b) follows from the fact that for a simple $C$-motive $\ul\CM$ the endomorphism algebra $E=\QEnd(\ul\CM)$ is a division algebra. (b) $\Rightarrow$ (c) follows from \cite[Corollaire 7.6/4]{BourbakiAlgebra} and the fact that $Q_\nu/Q$ is separable. (c) $\Rightarrow$ (d) follows from \cite[Corollaire de Proposition~6.4/9]{BourbakiAlgebra}. As $Q_\nu/Q$ is separable and $F_\nu=F\otimes Q_\nu$, one can argue that $(d)$ and $(e)$ are equivalent by \cite[Corollaire~7.6/4]{BourbakiAlgebra}. Also if $\pi_\nu$ is semi-simple then $\End_{Q_\nu[\Gamma_L]}(\omega_{Q_\nu}^\nu(\ul\CM))=\bigoplus_\mu Mat_{m_\mu\times m_\mu} (K_\mu)$, see Remark \ref{Rem_HomDec}. Thus $(d)$ implies $(c)$. Again we argue that $(c)$ implies $(b)$ by \cite[Corollaire~7.6/4]{BourbakiAlgebra}.\\
It remains to justify $(b)\Rightarrow (a)$. First observe that we can reduce to the this statement that if $\QEnd(\ul\CM)$ is a division algebra, then $\ul\CM$ is simple. To see this, suppose $\QEnd(\ul\CM)=\oplus_{j=1}^m Mat_{r_j\times r_j}(E_j)$ be the decomposition to the matrix algebras over division algebras $E_j$ over $Q$. Let $\{e_{j,i_j}\}_{1 \leq i_j \leq r_j}$ be the corresponding set of idempotents with $\sum e_{j, i_j}=\id_{r_j}\in Mat_{r_j\times r_j}(E_j)$ and $e_{j, i_j} \QEnd(\ul\CM) e_{j, i_j}=E_j$ . Now consider the quasi-isogeny $\sum_{i,j}e_{j,i_j}: \ul\CM \to\oplus_{i,j}  \ul\CM_{i_j}$, where $\ul\CM_{i_j}:=\im e_{j,i_j}$ and observe that $\QEnd(\ul\CM_{i_j})=e_{j,i_j} \QEnd(\ul\CM) e_{j,i_j}=E_j$. \\  
\noindent
Now assume that $E:=\QEnd(\ul\CM)$ is a division algebra. We show that $\ul\CM$ has no non-trivial quotient. Let $\ul\CM'$ be a quotient of $\ul\CM$ under $f:\ul\CM\to \ul\CM'$. Since $E$ is a division algebra, it is enough to show that there is an element $g\in \QHom(\ul\CM',\ul\CM)$ such that $g\circ f\neq 0$. 
\noindent
According to Proposition \ref{Prop_RationalTateisexactII}, the realization functor $\omega_{Q_\nu}^\nu(-)$ is exact, and thus by applying $\omega_{Q_\nu}^\nu(-)$ we get a surjection $$\omega_{Q_\nu}^\nu(f):\omega_{Q_\nu}^\nu(\ul\CM)\twoheadrightarrow \omega_{Q_\nu}^\nu(\ul\CM').$$ Note that as we have seen above $F_\nu$ is semi-simple and thus $\omega_{Q_\nu}^\nu(\ul\CM)$ is a finite semi-simple $F_\nu$-algebra. This implies that there exist 
$$
f_\nu'\in \Hom_{Q_\nu[\Gamma_L]}(\omega_{Q_\nu}^\nu(\ul\CM'),\omega_{Q_\nu}^\nu(\ul\CM))$$ 
such that $f_\nu \circ f_\nu'=\id_{\omega_{Q_\nu}^\nu(\ul\CM')}$. For sufficiently large integer $n$ we have 

$$
h_\nu:= z_\nu^n \cdot f_\nu' \in \Hom(\omega^\nu(\ul\CM'),\omega^\nu(\ul\CM))\cong \Hom_L(\ul\CM',\ul\CM)\otimes_A A_\nu.
$$

\noindent
Here $z_\nu$ is a uniformizer of $A_\nu$. Now take $g \in \Hom_L(\ul\CM',\ul\CM)$ such that $g \equiv h_\nu~(mod~ \nu^m)$ for sufficiently large $m>n$. We claim that $g\circ f \in E^\times$. Namely, if  $g\circ f=0$ then since $f$ is surjective we also have $f\circ g=0$ in  $\QEnd(\ul\CM')$. This implies that 

$$
z_\nu^n \cdot \id_{\omega_{Q_\nu}^\nu(\ul\CM)}=z_\nu^n  \cdot f_\nu \circ f_\nu'=f_\nu \circ h_\nu \equiv f\circ g=0 ~(mod~\nu^m), 
$$
which is a contradiction. Hence we argue that $f$ is also injective.

%According to Theorem \ref{Thm_Tate_Conj} this morphism induces a morphism $f'\in \Hom(\ul\CM',\ul\CM)$. Precisely we have $f'\circ f \neq 0$.
%Now take an integer $m$ enough big such that $\varpi^m.f_\nu'$ lies in $\Hom_{A_\nu[\Gamma_L]}(\omega^\nu(\ul\CM'),\omega^\nu(\ul\CM))$.        

\end{proof}

\noindent
Let $K$ be a field and $f, g \in K[x]$. Consider the factorizations $f=\prod \mu^{n_\mu}$ and $g=\prod \mu^{m_\mu}$, where $\mu$ runs over irreducible polynomials in $K[x]$. Set $r_K(f,g):=\prod_\mu m_\mu.n_\mu.\deg \mu$.

\begin{proposition}\label{Prop_EndomorphismRingI}
Let $\ul\CM$ be a $C$-motive over a finite field $L=\BF_{q^e}$ with semi-simple Frobenius $\pi$. Then $F=Q[\pi]$ is the center $\CZ(E)$ of the semi-simple $Q$-algebra $E=\QEnd_\ell(\ul\CM)$.   
\end{proposition}

\begin{proof}
Since $F_\nu$ is semisimple, the $F_\nu$-module $\omega^\nu(\ul\cM)$ is semisimple; see Theorem \ref{Thm_SemisimplicityQEnd}. As $\omega^\nu(\ul\CM)$ is finitely generated module over $E_\nu$ which itself is finite dimensional over $Q_\nu$. Therefore we have $Bicom_{F_\nu}(\omega^\nu(\ul\CM))=F_\nu$; see Remark \ref{RemBicom}\ref{RemBicomB}. Hence $\CZ(E_\nu)=E_\nu\cap F_\nu=F_\nu$ and $F\otimes_{Q}Q_\nu=F_\nu=\CZ(E\otimes_QQ_\nu)=\CZ(E)\otimes_QQ_\nu$, see Remark \ref{RemBicom}\ref{RemBicomA}. We conclude that $\dim_Q F=\dim_Q \CZ(E)$. Note that $F\subseteq \CZ(E)$, since for every $f\in E$ we have $f\circ \tau_\CM=\tau_\CM\circ f$ and thus $f\circ \pi=\pi\circ f$ .
\end{proof}

\begin{proposition}
Let $\ul\CM$ and $\ul\CM'$ be $C$-Motives over $L$ and let $V_\nu:=\omega_{Q_\nu}^\nu(\ul\CM)$ and $V_\nu':=\omega_{Q_\nu}^\nu(\ul\CM')$, at a place $\nu\in C$, different from characteristic places $\nu_i$. Assume further that $\pi_\nu \in \End_{Q_\nu[\Gamma]}(V_\nu)$ (resp. $\pi_\nu'\in \End_{Q_\nu[\Gamma]}(V_{\nu'})$) is semisimple. Then the dimension of $\QHom(\ul\CM,\ul\CM')$ as a $Q$-vector space equals $r_{Q_\nu}(\chi_\nu,\chi_\nu')$. 
\end{proposition}

\begin{proof}
Consider the decomposition $\chi_\nu=\prod_\mu \mu^{m_\mu}$ (resp. $\chi_\nu'=\prod_\mu \mu^{m_\mu'}$) of the characteristic polynomial $\chi_\nu$ (resp. $\chi_\nu'$) to the irreducible factors and set $K_\mu:=Q_\nu[x]/\mu$. Then decompose $V_\nu \cong \bigoplus_\mu (K_\mu)^{m_\mu}$ and $V_\nu' \cong \bigoplus_\mu (K_\mu)^{m_\mu'}$. We get

$$
\Hom_{Q_\nu[\Gamma_L]}(V_\nu,V_\nu')\cong\bigoplus_i Mat_{m_\mu'\times m_\mu} (K_\mu).
$$ 
\noindent
Now the assertion follows from Theorem \ref{Thm_Tate_Conj}.
\end{proof}

\begin{definition}\label{Def_CM}
We say that $\ul\CM$ has \emph{complex multiplication} if $\QEnd (\ul\CM)$ contains a commutative, semi-simple $Q$-algebra of dimension $\rk \ul\CM$. %Here semi-simple means that $E$ is a product of fields.

\end{definition}

\begin{proposition}\label{Proph|r}

Let $\ul\CM$ be a $C$-motive of rank $r$ over $L$ with Frobenius endomorphism $\pi$. Set $E:=\QEnd(\ul\CM)$.  Assume that $F = Q[\pi]$ is a field and let $h := [F : Q] = \deg \mu_\pi$. Then 
\begin{enumerate}

\item
$h|r$ and $\dim_Q E = r^2/h$ and $\dim_F E = r^2/ h^2$. 
\item
For any place $\nu$ of $Q$ different from characteristic places $\nu_i$ we have $E\otimes Q_\nu \cong Mat_{r/h\times r/h}(F \otimes_Q Q_\nu)$ and $\chi_\nu = (\mu_\pi)^{r/h}$, independent of $\nu$.

\end{enumerate}
\noindent
In particular if $\ul\CM$ is CM then $F=E=\QEnd(\ul\CM)$ is commutative.

\end{proposition}

\begin{proof}

Since $F$ is a field, $\pi_\nu$ is semi-simple, see Theorem \ref{Thm_SemisimplicityQEnd}. Therefore the minimal polynomial $\mu_{\pi_\nu}$ equals $\prod_\mu \mu$, with pairwise different monic irreducible polynomials $\mu\in Q_\nu[x]$. The characteristic polynomial $\chi_\nu$ then equals $\prod_\mu \mu^{m_\mu}$. We have $E_\nu\cong \prod_\mu E_\mu$, where $E_\mu=Mat_{m_\mu\times m_\mu}(K_\mu)$ and $K_\mu=Q_\nu[x]\slash \mu$, see Remark~\ref{Rem_HomDec}. We get $Q_\nu[\pi_\nu]=F_\nu=F\otimes_Q Q_\nu\twoheadrightarrow Q_\nu[x]/(\mu)= K_\mu$ and the surjection
$$
E_\nu\otimes_{F_\nu} K_\mu =(E\otimes_Q Q_\nu)\otimes_{F_\nu}K_\mu=E\otimes_F(F\otimes_QQ_\nu)\otimes_{F_\nu}K_\mu=E\otimes_F K_\mu \twoheadrightarrow E_\mu
$$

\noindent
In particular $m_\mu^2\leq \dim_F E$. Thus

$$
\CD
\dim_F E \cdot [F:Q]=\dim_QE=\dim_{Q_\nu} E_\nu=\sum_\mu m_\mu^2\cdot\deg_x \mu\leq (\dim_F E) \cdot \sum_\mu \deg_x \mu\\
=(\dim_F E)\cdot \deg_x\mu_{\pi_\nu}=[F:Q]\cdot\dim_F E
\endCD
$$
Therefore $m_\mu^2=\dim_F E$ for every $\mu$ and 
$$
r=\deg_x\chi_\nu=\sum_\mu m_\mu \deg_x \mu=\sqrt{\dim_F E}~\cdot~\sum_\mu \deg_x \mu=\sqrt{\dim_F E}~\cdot~[F:Q].  
$$
Therefore $r=m_\mu\cdot h$ and $\dim_F E=r^2/h^2$. Now to see part (b) write $E_\nu\cong\oplus_\mu Mat_{r/h\times r/h}(K_\mu)\cong Mat_{r/h\times r/h}(\oplus_\mu K_\mu)\cong Mat_{r/h\times r/h}(F_\nu)$. 
\end{proof}

\noindent
%-\textbf{Endomorphism $Q$-algebra} 
Let us state the following proposition, regarding the two extreme cases for $F\subseteq E$.

\begin{proposition}\label{Prop_ExtremeCases}

Let $\ul\CM$ be a $C$-motive over $L$ with semisimple Frobenius endomorphism $\pi:=\pi_{\ul\CM}$, i.e. $F=Q[\pi]$ is a product of fields. Let $\nu$ be a place on $C$ apart from characteristic places $\nu_i$. Let $\chi_\nu$ denotes the characteristic polynomial of $\pi_\nu:=\omega^\nu(\pi)$. We  have the following statements

\begin{enumerate}
\item
$F=Q(\pi)$ is the center of the semisimple $Q$-algebra $E=\QEnd_L(\ul\CM)$.
\item
$\rk \ul\CM\leq [E:Q]:=\dim_QE\leq (\rk\ul\CM)^2$
\item
The following are equivalent
\begin{enumerate}

\item[i)] $E=F$
\item[ii)]
$E$ is commutative,
\item[iii)]
$[F:Q]=\rk\ul\CM$ 
\item[iv)]
$[E:Q]=\rk\ul\CM$
\item[v)]
$\chi_\nu$ is product of pairwise different irreducible polynomials in~$Q_\nu[x]$ 
\end{enumerate}

\item

The following are equivalent\begin{enumerate}

\item[i)]
$F=Q$
\item[ii)]
$E\cong Mat_{n\times n}(D), \text{for a division algebra $D$ with center $\CZ(D)=Q$}$ 
\item[iii)] $[F:Q]=1$
\item[iv)]$[E:Q]=(\rk\ul\CM)^2$
\item[v)]$\chi_\nu=\mu^{\rk\ul\CM} \text{for a linear polynomial}~ \mu\in Q_\nu[x]. \text{This is the minimal polynomial $\mu_\pi$}.
$

\end{enumerate}

\end{enumerate}

\end{proposition}

\begin{proof}
(a) was proved in Proposition \ref{Prop_EndomorphismRingI}. For (b) let $\chi_\nu=\prod_\mu\mu^{m_\mu}$ with $\mu\in Q_\nu[x]$ irreducible pairwise different. We have the decomposition $E\otimes_Q Q_\nu\cong\prod_\mu Mat_{m_\mu\times m_\mu}(Q_\nu[x]/(\mu))$ and thus
$$
[E:Q]=\dim_Q E=\sum_\mu m_\mu^2\cdot \deg_x \mu
$$
and
$$
\sum_\mu m_\mu\cdot \deg_x \mu=\deg_x\chi_\nu=\dim_{Q_\nu} \omega_Q^\nu(\ul\CM)=\rk \ul\CM.
$$ 
Therefore 

\begin{equation}\label{Eq_InequalityI}
\rk\ul\CM=\sum_\mu m_\mu\cdot \deg_x \mu \leq \sum_\mu m_\mu^2\deg_x \mu=[E:Q]\leq(\sum_\mu m_\mu\deg_x \mu)^2=(\rk\ul\CM)^2.
\end{equation}

%THE SIMPLIFIED PROOF
Let us prove part c). 
The implications $i)\Leftrightarrow ii)$ follows from (a).  As we have seen above $[E:Q]=\rk\ul\CM$ if and only if $m_\mu=1$ for all $\mu$. This shows that  $$iv)\Leftrightarrow v)\Rightarrow iii)$$  
We know that $\mu_\pi$ is the minimal polynomial of $\pi_\nu:=\omega^\nu(\pi)$ and since divides the characteristic polynomial, we argue that $\mu_\pi=\chi_\nu$. Hence $iii)\Rightarrow v)$. Now if we have $iii)$ then we also have $iv)$, and as $F\subseteq E$ we thus see that $E=F$. Finally $i)$ implies that $E$ is a product of fields therefore we may conclude that $i) \Rightarrow v)$ by Remark \ref{Rem_HomDec}.

\forget{
--------------------------\\
\comment{The previous proof of part c)}

We proceed the proof of part c) in the following way
$$
i)\Leftrightarrow ii)\Rightarrow  iii)\Rightarrow i) \Rightarrow iv)\Leftrightarrow v)\Rightarrow iii).
$$

The first implication follows from (a).  As we have seen above $[E:Q]=\rk\ul\CM$ if and only if $m_\mu=1$ for all $\mu$. This implies that $[F:Q]=\rk\ul\CM$. We know that $\mu_\pi$ is the minimal polynomial of $\pi_\nu:=\omega^\nu(\pi)$ and since divides the characteristic polynomial, we argue that $\mu_\pi=\chi_\nu$. Hence 
$$
[F:Q]=\deg_x\mu_\pi=\deg_x\chi_\nu=\dim_{Q_\nu}\omega^\nu(\ul\CM)=\rk\ul\CM.
$$
 
\noindent 
Let's show that $E=F$ if and only if $[E:Q]=\rk\ul\CM$. We have

$$
\rk\ul\CM=\deg_x \chi_\nu\geq \deg_x \mu_\pi=[F:Q]=[E:Q]\geq\rk\ul\CM,
$$
conversely,
$$
\sum_\mu \deg_x \mu=\deg_x\mu_\pi=\rk\ul\CM=\deg_x\chi_\nu=\sum_{i=1}^n m_\mu\cdot \deg_x \mu
$$
Therefore $m_\mu=1$, for all $\mu$ and thus $[E:Q]=\rk\ul\CM=[F:Q]$, which implies $E=F$.
-----------------------------\\
}

\noindent
It remains to prove $(d)$. If $F=Q$, then by a) the center of $E$ is a field, and therefor ii) follows from Artin-Wedderburn theorem. Conversely if $E\cong Mat_{n\times n}(D)$ then $F=\CZ(E)=\CZ(Mat_{n\times n}(D))=\CZ(D)=Q$. The equivalence $[E:Q]=(\rk\ul\CM)^2 \Leftrightarrow \chi_\nu=\mu_\pi^{\rk\ul\CM}$ follows from \ref{Eq_InequalityI}. Finally, assuming $v)$, we see by semi-simplicity of $\pi$ that $\mu_\pi=\mu$ and therefore
$
[F:Q]=\deg_x\mu_\pi=\deg_x\mu=1. 
$
On the other hand if $\mu=\mu_\pi$ is linear then $\chi_\nu=(\mu_\pi)^{\rk\ul\CM}$ by Cayley-Hamilton. 
   
\end{proof}

\section{Relation to G-Shtukas}\label{SectGSht}

Let $\FG$ be a flat affine group scheme of finite type over the curve $C$ with generic fiber $G$.  
Let $\scrH^1(C,\FG)$ denote the stack whose $S$ points parameterize $\FG$-bundles on $C_S:=C\times_{\BF_q} S$. 

%This stack is an Artin-stack locally of finite type over $\BF_q$. See \cite[Theorem~2.5]{AH_Global}.

%\begin{theorem}
%The stack $\scrH^1(C,\FG)$ is a smooth Artin-stack locally of finite type over $\BF_q$. It admits a covering by connected open substacks of finite type over $\BF_q$.
%\end{theorem}

\subsection{$G$-motives and functoriality}

One may endow the category of $C$-motives $\mot(S)$ with a $G$-structure. This leads to the following definition.

\begin{definition}\label{DefGMotives}
Let $\Rep \FG$ denote the category of representations of $\FG$ in finite free $\cO_C$-modules $\CV$.  By a \emph{$\FG$-motive} (resp. \emph{$G$-motive}) over $S$ we mean a tensor functor $\ul \CM_\FG: \Rep \FG \to \mot(S)$ (resp. $\ul \CM_G: \Rep_Q~G \to \mot(S)$). We say that two $\FG$-motives (resp. $G$-motives) are \emph{isomorphic} if they are isomorphic as tensor functors. We denote the resulting category of $\FG$-motives (resp. $G$-motives) over $S$ by $\FG$-$\mot(S)$ (resp. $G$-$\mot(S)$).
\end{definition}

\noindent
Note that the construction of the category
$G$-$\mot(S)$ is functorial both in $G$ and $C$. Namely\\

\begin{enumerate}
\item[-] morphism $\rho: G\to G'$ (resp. $\FG\to\FG'$) induces a functor $\Rep G' \to \Rep G$ (resp. $\Rep \FG' \to \Rep \FG$) and this further induces a functor $\rho_\ast: G\text{-}\mot(S)\to G'\text{-}\mot(S)$ (resp. $\rho_\ast: \FG\text{-}\mot(S)\to \FG'\text{-}\mot(S)$).

\item[-] Suppose we have a morphism $C\to  C'$ of smooth projective geometrically irreducible curves which is of degree $d$. The characteristic places $\ul\nu$ on $C$ induce an n-tuple of characteristic places on $C'$, which we denote by $\ul\nu'$. In addition the push forward functor $f_\ast$ induces a functor from the category of locally free sheaves of rank $r$ over $C_S$ to the category of locally free sheaves of rank $r\cdot d$ over $C_S'$. This further  induces the push forward functor $f_\ast: G\text{-}\mot(S)\to G\text{-}\CM ot_{C'}^{\ul\nu'}(S)$.  \\

\end{enumerate}

%Clearly the above push forward functors are compatible with realization functors.

\subsection{$G$-Shtukas and functoriality}

Let's now discuss the geometrization of the   category $G$-$\mot(S)$. For this we first recall the following definition of the moduli (stacks) of global $\FG$-shtukas. 

\begin{definition-remark}\label{Def-RemGlobal Sht}
\begin{enumerate}
\item
A \emph{global $\FG$-shtuka} $\ul\CG$ over an $\BF_q$-scheme $S$ is a tuple $(\CG,\ul\charsect,\tauGlob)$ consisting of 
\begin{enumerate}
\item[-]
a $\FG$-bundle $\CG$ over $C_S$, 
\item[-]
an $n$-tuple $\ul\charsect$ of (characteristic) sections and 
\item[-]
an isomorphism $\tauGlob\colon  \s \CG|_{C_S\setminus \Gamma_{\ul\charsect}}\isoto \CG|_{C_S\setminus  \Gamma_{\ul\charsect}}$.

\end{enumerate}

We let $\nabla_n\scrH^1(C,\mathfrak{G})$ denote the stack whose $S$-points parameterizes global $\FG$-shtukas over $S$. Sometimes we will fix the sections $\ul s:=(\charsect_i)_i\in C^n(S)$ and simply call $\ul\CG=(\CG,\tauGlob)$ a global $\FG$-shtuka over $S$.

\forget
{
--------------------------------
\item\label{RemNablaHisIndDM}

It can be shown that the stack $\nabla_n \scrH^1(C,\FG)$ is an ind-Deligne-Mumford stack over $C^n$ which is ind-separated and locally of ind-finite type. Indeed, a choice of faithful representation of $\FG$ in $\SL_r$ induces an ind-structure 
$$
\nabla_n \scrH^1(C,\FG):=\dirlim[\ul\omega]\nabla_n^\ul\omega \scrH^1(C,\FG)
$$ 
such that $\nabla_n^\ul\omega \scrH^1(C,\FG)$ are Deligne-Mumford. Here $\ul\omega$ runs over n-tuple of coweights of $\SL_r$. For the proof see \cite[Theorem 3.15]{AH_Global}. %The construction of the ind-algebraic structure is also recalled in appendix \ref{App.B}. 
----------------------------------
}

\item 
We let $\nabla_n\scrH^1(C,\FG)(S)_Q$ denote the category which has the same objects as $\nabla_n\scrH^1(C,\FG)(S)$, but the set of morphisms is enlarged to \emph{quasi-isogenies} of $\FG$-shtukas. A quasi-isogeny  $f:\ul\CG\to\ul\CG'$ is a commutative diagram

$$
\CD
\CG@>f>>\CG'\\
@A{\tau}AA @AA{\tau'}A\\
\sigma^\ast\CG @>\sigma^\ast f >>\sigma^\ast\CG',
\endCD
$$
defined outside $\Gamma_\ul s \cup D\times_{\BF_q}S$ for a closed subscheme $D\subseteq C$. Note that as a morphism of torsors $f$ is automatically an isomorphism. We denote by $QIsog_S(\ul\CG,\ul\CG')$ the set of quasi-isogenies between $\ul\CG$ and $\ul\CG'$.

\item
We denote by $\nabla_n\scrH^1(C,\FG)^{\ul\nu}$, the formal stack which is obtained by taking the formal completion of the stack $\nabla_n\scrH^1(C,\FG)$ at a fixed $n$-tuple of pairwise different characteristic places $\ul\nu=(\nu_1,\ldots,\nu_n)$ of $C$ in the sense of  \cite[Definition A.12]{Har1}. This means we let $A_{\ul\nu}$ be the completion of the local ring $\CO_{C^n,\ul\nu}$, and we consider global $\FG$-shtukas only over schemes $S$ whose characteristic morphism $S\to C^n$ factors through $\Nilp_{A_\ul\nu}$. We similarly denote by $\nabla_n\scrH^1(C,\FG)^{\ul\nu}(S)_Q$ the category over $S$, with morphisms enlarged to quasi-isogenies.

\end{enumerate}
\end{definition-remark}

\forget
{
----------------------
Let us recall the following feature of the quasi-isogenies between global $\FG$-shtukas. Namely, they have rigidity property, in the sense that a quasi-isogeny lifts over infinitesimal thickenings. The proof was given in \cite[Proposition~5.9]{AH_Local}. Also see \cite{Drinfeld2} for the corresponding fact for p-divisible groups (and abelian varieties).

\begin{theorem}[Rigidity of quasi-isogenies for global $\FG$-shtukas]\label{Thm_Rigidity_G_Sht} Let $\ul\nu$ be a set of n distinct places on $C$. Let $S$ be a scheme in $\Nilp_{\wh\CO_{C^n,\ul\nu}}$ and let $j : \ol S \to S$ be a closed immersion defined by a sheaf of ideals $\CI$ which is locally nilpotent. Let $\ul\CG = (\CG,\tau)$ and $\CG' = (\CG',\tau')$ be two global $\FG$-shtukas over $S$. Then 
$$
QIsog_S(\ul\CG,\ul\CG') \to QIsog_S(j^\ast \ul\CG,j^\ast \ul\CG'),~ f\mapsto j^\ast f 
$$
is a bijection of sets. Moreover, $f$ is an isomorphism at a place $\nu\notin \ul\nu$ if and only if $j^\ast f$ is an isomorphism at $\nu$. 
\end{theorem}

-----------------------------------
}

\begin{remark}\label{RemMotisGLnShtuka}
The category of vector bundles of rank $n$ on $C_S$ can be identified with the category of $\GL_n$-bundles on $C_S$. Consequently, one can identify the category $\nabla_n\scrH^1(C,\GL_n)^\ul\nu(S)_Q$ with the subcategory $C$-motives $\mot(L)^\circ$ of rank $n$. %In particular the above rigidity property of quasi-isogenies also holds for $C$-motives.
\end{remark}

The assignment of the moduli stack $\scrH^1(C,\FG)$ to a group $\FG$, is functorial. In other words a morphism $\rho:\FG\to\FG'$ of algebraic groups gives rise to a morphism 

$$
\CD
\scrH^1(C,\FG)\to\scrH^1(C,\FG'),\\
~~~~~~~~~~~~~~~~~~\CG\mapsto \rho_\ast\CG:=\CG\times_{\FG,\rho} \FG'
\endCD
$$

\noindent
In particular any representation $\rho:\FG\to\GL(\CV)$ induces a natural morphism of $\BF_q$-stacks

\begin{equation}\label{Eq_pushforwardalongrepresentations}
\scrH^1(C,\FG)\to \scrH^1(C,\GL(\CV)),\\\newline
\CG\mapsto \rho_\ast \CG.
\end{equation}

\noindent
Note that by tannakian theory having a $\FG$-bundle over $C_L$ is equivalent with giving a tensor functor $f$ from $\Rep \FG$\ to the category $\Vect_C(L)$ of vector bundles over $C_L$.

Regarding the functoriality of $\scrH^1(C,-)$, the assignment of the moduli stack $\nabla_n\scrH^1(C,\FG)$ of global $\FG$-shtukas to a group $\FG$ is functorial in $\FG$, i.e. a morphism $\FG\to\FG'$ gives rise to the morphism
$$
\rho_\ast(-): \nabla_n\scrH^1(C,\FG)\to \nabla_n\scrH^1(C,\FG').
$$
Given a $\FG$-shtuka $\ul\CG$ and a representation $\ol\rho:G\to \GL(V)$ of the group $G$ on a $Q$-vector space, we may consider a lift $\rho:\FG\to\GL(\CV)$ and use it to push forward $\ul\CG$ to produce a $\GL(\CV)$-shtuka $\rho_\ast(\ul\CG)$ in $\nabla_n\scrH^1(C,\GL(\CV))(L)$. According to Remark \ref{RemMotisGLnShtuka}, $\rho_\ast(\ul\CG)$ can be viewed as a $C$-motive in $\mot(L)$. Note that taking a different lift $\rho'$ of the representation $\ol\rho$, $\rho_\ast'\CG$ and $\rho_\ast\ul\CG$ are canonically isomorphic in $\mot(L)$. This is because $\rho$ and $\rho'$ agree over an open $U\subset C$.
In particular we obtain the following pairing 

\begin{equation}\label{EqPairingGShtandGMot}
\nabla_n\scrH^1(C,\FG)^\ul\nu(L)_Q\times \Rep_Q G \to \mot(L).
\end{equation}
\begin{definition}
The above pairing \ref{EqPairingGShtandGMot} induces a functor 

\begin{equation}\label{Eq_FunctShtToGMot}
\ul\CM_G^{-}: \nabla_n\scrH^1(C,\FG)^\ul\nu(L)_Q \;\longto\; G\text{-}\mot(L), \qquad
\ul\CG\mapsto \ul\CM_G^{\ul\CG}
\end{equation}

\noindent
We say that $\ul\CM_G$ in $G\text{-}\mot(L)$ is \emph{geometric} if it arises from a $\FG$-shtuka $\ul\CG$ via the above functor. Note that this functor may fail to be essentially surjective for general $\FG$.

\end{definition}

\forget{
\begin{remark}\label{RemEqGSht=GMot}
The above pairing is a perfect pairing in the sense that it induces an equivalence of categories $\nabla_n\scrH^1(C,\FG)^\ul\nu(L)_Q \tilde{\to} ~~G\text{-}\mot(L)$. To construct the quasi-inverse functor, choose a lift $\Rep G \to \cM ot^{\ul\nu}(\BF)$, by tannakian formalism this yields a global $\FG$-shtuka $\ul\CG$. One can observe that a different choice for the lift yields a $\FG$-shtuka $\ul\CG'$ that differs by a unique quasi-isogeny $\ul\CG'\to \ul\CG$. 
\end{remark}

\comment{Referee's comment:\\
``Why is it an equivalence? the right-hand side is a category of G-bundles, while the other is that of $\FG$-bundles, so an extension from G-bundles to $\FG$-bundles seems necessary. "\\

I add few lines to illustrate it better, but still can not s sure that the referee will find it very enlightening; maybe you see if we can still explain it better?}

}

\forget{
Recall that a motive which becomes mixed Tate over an algebraic closure, is already mixed Tate over a finite separable extension; see Proposition \ref{Prop_GeomMixedTate areMixedTateOverFiniteExtI} in the appendix A. Similarly a geometric motive $M$ in $DM_{gm}^{eff} (F^\sep)$ comes from a geometric motive over some finite field extension $E/F$. As a consequence of the geometrization of the category $G$-$\mot(S)$, one can immediately see that a similar fact also holds in the category of ($G$-)$C$-motives. 
}

Recall that, for a perfect field $F$ a motive $M$ in the category $DM_{gm}^{eff} (F^\sep)$ of effective geometric motives, see \cite{VSF} for the definition and notation, comes from a geometric motive over some finite field extension $E/F$. As a consequence of the geometrization of the category $G$-$\mot(S)$, one can immediately see that a similar fact also holds in the category of ($G$-)$C$-motives.

\begin{proposition}\label{Prop_GeomMixedareMixedOverAFiniteExtensionII}
Let $F$ be a field over $\BF_q$. Any geometric $G$-$C$-motive $\ul\CM_G$ in $G$-$\mot(F^\sep)$ comes by base change from a $G$-$C$-motive in $G$-$\mot(E)$ for a finite extension $E/F$. 
\end{proposition}

\begin{proof}
 
\forget{ 
 The proof goes in a different manner than \ref{Prop_GeomMixedTate areMixedTateOverFiniteExtI}. Namely, this observation can be made by looking at $\nabla_n\scrH^1(C,\FG)$ as a moduli space for motives, regarding equivalence of categories \eqref{EqGSht=GMot}, and then using the existence of ind-algebraic structure on $\nabla_n\scrH^1(C,\FG)$, which we explain in this section, see Theorem \ref{ThmNablaHisIndAlg} below. The proposition immediately follows.\\
 
ALTERNATIVE PROOF:
}

First let us assume that $\ul\CM_G$ arises from $\ul\CG=(\CG,\tau)$ under (\ref{Eq_FunctShtToGMot}) and then observe that any $\FG$-bundle $\CG$ over $C_{F^\sep}$ is a scheme of finite type over $C_{F^\sep}$. To see this, one can take an \fppf-covering $U'\to C_{F^\sep}$ which trivializes the $\FG$-bundle $\CG$. Therefore we observe that $\CG|_{U'}\cong \FG\times_CU'$ is of finite type over $U'$, and since $U'\to C_{F^\sep}$ is an \fppf-cover, we may argue that $\CG$ is of finite type by \fppf-descent \cite[VI 2.7.1]{EGA}. As for $\ul\CG:=(\CG,\tau)$ in $\nabla_n\scrH^1(C,\FG)(F^\sep)$, the morphism $\tau$ is defined between schemes of finite type, we see that $\ul\CG$ comes from a point in $\nabla_n\scrH^1(C,\FG)(E)$ for some finite extension $E/F$. This defines the corresponding $G$-$C$-motive in $G$-$\mot(E)$. Note that one can alternatively prove this statement, using the ind-algebraic structure on  $\nabla_n\scrH^1(C,\FG)$ constructed in \cite[Theorem 3.15]{AH_Global}. %See Theorem \ref{ThmNablaHisIndAlg} below.
 
\end{proof}

\subsection{Quasi-Isogenies and Galois Stable Sublattices}\label{Subsec_Global-Local Functor}

Like abelian varieties we can pull back a C-motive along a morphism to the corresponding local iso-crystal. %Indeed, this holds more generally for $\FG$-shtukas and associated local $\BP$-shtukas.

\begin{proposition}\label{Prop_Pullback}
Fix a place $\nu$ on $\dot{C}$. Let $\ul\CM$ be a $C$-motive over $L$. Let $\hat{\ul M}$ denote the crystal $\Gamma_\nu(\ul\CM)$ associated to $\ul\CM$ at the place $\nu$.  For a given quasi-isogeny (i.e. a morphism which is generically an isomorphism) $\hat{f}: \hat{\ul M}'\to \hat{\ul M}$ of \'etale crystals, one can construct a unique $C$-motive $\ul\CM'$ with a unique quasi-isogeny $f: \ul\CM'\to\ul\CM$ such that it gives $\hat{f}$ back, after applying the functor $\Gamma_\nu(-)$. We denote $\ul\CM'$ by $\hat{f}^\ast \ul\CM$.

\end{proposition}

\begin{proof}
Recall that for a noetherian ring $A$, the $I$-adic completion $\hat{A}$ of $A$, with respect to an ideal $I\subset A$, is flat over $A$. Thus the natural maps 
$$
(C \setminus \{\nu\})\sqcup \Spec(\hat{\CO}_{C,\nu}) \to C ~\text{and}~(C \setminus \{\nu\})_L\sqcup \Spec(\hat{\CO}_{C,\nu}\hat{\otimes} L) \to C_L
$$
are \fpqc-covers. Let $z_\nu$ be a uniformizer of $\hat{\CO}_{C,\nu}$. Consider the following Cartesian diagram
$$
\CD
\Spec(\hat{\CO}_{C,\nu}\hat{\otimes} L[1/z_\nu])@>j>> (C\setminus \{\nu\})_L\\
@VVV @VVV\\
\Spec(\hat{\CO}_{C,\nu}\hat{\otimes}L) @>>>C_L.
\endCD
$$
By the \fpqc-descent of quasi-coherent sheaves, to give a vector bundle $\CM$ over $C_L$ is the same as to give a triple $(\CN,\hat{M},\iota)$, consisting of a vector bundle $\CN$ over $(C \setminus\{\nu\})_L$, a finite projective $\hat{\CO}_{C,\nu}\hat{\otimes}L$-module $\hat{M}$ and an isomorphism $\iota : j^\ast \CN \to \dot{\hat{M}}$, where the latter means the sheaf associated to the module $\hat{M}[1/z_\nu]$. For a given vector bundle $\CM$ over $C_L$ and a quasi-isogeny (i.e. a morphism which is generically isomorphism) $\hat{f}: \hat{M}'\to \hat{M}$ of finite projective $\hat{\CO}_{C,\nu}\hat{\otimes} L$-modules, the triple $(\CM|_{(C\setminus\{\nu\})_L},\hat{M}',(\hat{f}[1/z_\nu])^{-1}\circ \iota)$ and the map $(\id,\hat{f})$ define a unique vector bundle $\CM'$ and a unique quasi-isogeny $f : \CM' \to \CM$ such that $f$ is isomorphism outside $\nu$ and gives $\hat{f}$ back after passing to the completion at $\nu$.
Similarly, for a $C$-motive $\CM$ over $L$, the map $\tau$ of $\ul\CM:=(\CM,\tau)$ can be interpreted as a map between triples 
$$
(\sigma^\ast \CM|_{(C\setminus\{\nu,\ul\nu\})_L},\sigma^\ast \hat{M}, \sigma^\ast\iota)\to (\CM|_{(C\setminus\{\nu,\ul\nu\})_L}, \hat{M},\iota).
$$
Thus, for a given morphism $\hat{f}: \ul{\hat{M}}' \to \ul{\hat{M}}$ of the category $\textbf{\'Et~$(\sigma,\nu)$-Cryst}(L)$, the vector bundle $\CM'$ can be equipped with an isomorphism $\tau':\sigma^\ast\CM'|_{\dot{C}_L}\to\CM'|_{\dot{C}_L}$. This yields a $C$-motive $\ul\CM'=(\CM',\tau')$ over $L$, and moreover, the map $f$ gives rise to a quasi-isogeny $f: \ul\CM'\to\ul\CM$ of $C$-motives. By construction the quasi-isogeny $f$ gives $\hat{f}$ after applying $\Gamma_\nu(-)$. Now we may conclude by Proposition \ref{Prop_RedisIsom}. % we may view $\hat{f}$ as a morphism in the category $\textbf{\'Et~$(\sigma,\nu)$-Cryst}(L)$
\end{proof}

%\begin{remark}
%The above proposition is a particular case of \cite[Proposition 5.7]{AH_Local} for $\FG=\GL_r$; see also Remark \ref{RemMotisGLnShtuka}.%, Remark \ref{Remark_Local_GLr_Shtukas} and Remark \ref{RemGlobLocFunctandCrystallineRealization}.
%\end{remark}

\begin{corollary}\label{Cor_Qiso&SL}
Fix a place $\nu$ on $\dot{C}$. Let $\phi:\ul\CM' \to \ul\CM$ be a quasi-isogeny of $C$-motives in $\mot (L)$. Then $\omega_Q^\nu(\phi)$ identifies $\omega^\nu(\ul\CM')$ with a $\Gamma_L$-stable sublattice of $\omega_Q^\nu(\ul\CM)$. This gives a one to one correspondence between  the following sets

$$
\{ \text{quasi-isogenies $\ul\CM'\to\ul\CM$ in $\mot(L)$ which are isomorphisms above $\ul\nu$} \}  
$$

and

$$
\{\text{$\Gamma_L$-stable sublattice $\Lambda_\nu \subseteq \omega_Q^\nu(\ul\CM)$ which are contained in $\omega^\nu(\ul\CM)$} \}.
$$
\end{corollary}

\begin{proof}
Note that $\nu$ does not lie in the characteristic places $\ul\nu$, and hence the associated crystals $\Gamma_\nu(\ul\CM')$ and $\Gamma_\nu(\ul\CM)$ at $\nu$ are \'etale.
We may view  $\omega^\nu(\ul\CM')$ as a $\Gamma_L$-stable sublattice of $\omega_Q^\nu(\ul\CM)$ contained in $\omega^\nu(\ul\CM)$ by applying $\omega^\nu(-)$ to a given quasi-isogeny $f:\ul\CM'\to\ul\CM$

$$
\omega^\nu(\ul\CM')\hookrightarrow \omega^\nu(\ul\CM)\subseteq \omega_{Q_\nu}^\nu(\ul\CM).
$$
Vice versa, consider the inclusion 

$$
\Lambda_\nu\otimes_{A_\nu} L^\sep\dbl z\dbr\subseteq\omega^\nu(\ul\CM)\otimes_{A_\nu}L^\sep\dbl z\dbr=\Gamma_\nu(\ul\CM)\otimes_{L\dbl z\dbr} L^\sep\dbl z\dbr.
$$
Therefore we get a quasi-isogeny $\hat{f}$ from $\hat{\ul M}':=(\Lambda_\nu\otimes_{A_\nu} L^\sep\dbl z\dbr)^{\Gamma_L}$ to $\hat{\ulM}:=(\Gamma_\nu(\ul\CM)\otimes_{L\dbl z\dbr} L^\sep\dbl z\dbr)^{\Gamma_L}$. Note that both $\hat{\ulM}$ and $\hat{\ulM}'$ are \'etale crystals. Namely, $\hat{\ul M}$ is \'etale by Theorem~\ref{Thm_EmbeddingOfEtCrystintoGalMods}, and $\hat{\ul M}'$ is \'etale, because of the following reason. Since $\hat{ M}'\hookrightarrow \hat{M}$, we observe that the former module is also a finite projective $A_{\nu,L}$-module, moreover the Frobenius map on $\hat{M}'=(\Lambda_\nu\otimes_{A_\nu} L^\sep\dbl z\dbr)^{\Gamma_L}$ is given by $\id\otimes \hat{\sigma}_\nu^\ast$, so it is an isomorphism. Hence, according to Proposition \ref{Prop_Pullback}, we may construct the pull back $\ul\CM':=\hat{f}^\ast \ul\CM$ of $\ul\CM$ along $\hat{f}$, which comes with a canonical quasi-isogeny $f:\ul\CM'\to\ul\CM$. By construction the above assignments are inverse to each other.

\end{proof}

\section{Quasi-isogeny classes and Honda-Tate theory}\label{Sect_Quasi-isogeny classes and Honda-Tate theory}

Recall that a Weil $p^n$-number is an algebraic number $\pi$ for which there exists an integer $m$ such that $\pi\ol\pi=p^n$ for all $Q[\pi] \to \BC$. Here $\ol \pi$ denotes the complex conjugate of $\pi$. The Honda-Tate theory, \cite{Honda} and \cite{Tate66}, states that sending an abelian variety $\CA$ over a finite field with $q$-elements to the eigenvalue of Frobenius endomorphism $\pi_\CA$ on the first \'etale cohomolgy group, gives a bijection between isogeny classes of simple abelian varieties over $\BF_q$ and the set of Weil $p^n$-numbers $W(p^n)$ (up to conjugation). %See also Remark \ref{Rem_Honda-Tate_Theory} for the motivic version. 

\bigskip
In this section we discuss the analogous picture for $\mot(\ol\BF_q)$. Note that, unlike the above case of abelian varieties, as it is mentioned earlier, $C$-motives are not pure. This means that the eigenvalues of Frobenius endomorphism may have different valuations. So in particular one must modify the group of Weil (q-)numbers.  \\

\begin{proposition}\label{PropQIsoClassesI}
Let $L/\BF_q$ be a finite field. For $\ul\CM$ and $\ul\CM'$ in $\mot(L)$, we let $\pi:=\pi_\ul\CM$ and $\pi':=\pi_{\ul\CM'}$ denote the corresponding Frobenius endomorphism with minimal polynomials $\mu:=\mu_{\pi_\ul\CM}$ and $\mu':=\mu_{\pi_\ul\CM}$. Let $\chi_\nu$ and $\chi_\nu'$ denote the characteristic polynomials of $\pi_\nu$ and $\pi_\nu'$. Consider the following statements
\begin{enumerate}
\item
$\ul\CM'$ is quasi-isogenous to a quotient of $\ul\CM$. 
\item
$\omega_{Q_\nu}^\nu(\ul\CM')$ is $\Gamma_L$-isomorphic to a $\Gamma_L$-quotient space of $\omega_{Q_\nu}^\nu(\ul\CM)$. 
\item
$\chi_\nu'$ divides $\chi_\nu$ in $Q_\nu[x]$. 
\item
$\mu'$ divides $\mu$ in $Q[x]$ and $\rk \ul\CM \leq \rk \ul\CM'$.
\end{enumerate}
then (a) and (b) are always equivalent and imply (c) and (d). Furthermore we have the following statements
\begin{enumerate}

\item[i)]
If $\pi_\nu$ and $\pi_\nu'$ are semi-simple then (c) also implies (b).
\item[ii)]
If $\mu$ is irreducible, then all the above statements are equivalent.

\end{enumerate}
\end{proposition}

\begin{proof}
$(a) \Rightarrow (b)$ is obvious. So let us first show that $(b)\Rightarrow (a)$. The main ingredient to prove this is the analog of Tate conjecture; See Theorem \ref{Thm_Tate_Conj}. Consider the quotient morphism $f_\nu:\omega_{Q_\nu}^\nu(\ul\CM)\to\omega_{Q_\nu}^\nu(\ul\CM')$. Multiplying with a suitable power of the uniformizer $z_\nu\in A_\nu$, we may assume that it is defined with integral coefficients $f_\nu: \omega^\nu(\ul\CM)\to\omega^\nu(\ul\CM')$ with $$z_\nu^N\omega^\nu(\ul\CM')\subseteq f_\nu(\omega^\nu(\ul\CM)),$$ for some integer $N\gg 0$. By Theorem \ref{Thm_Tate_Conj} $f_\nu$ can be viewed as an element of $$\Hom_L(\ul\CM,\ul\CM')\otimes_A A_\nu$$ and thus induces a morphism $f:\ul\CM\to\ul\CM'$ such that $\omega^\nu(f)=f_\nu~(\mod~z_\nu^{N+1})$. We claim that $\dim_{Q_\nu}\omega_{Q_\nu}^\nu(f)(\omega^\nu (\ul\CM))=r':=\rk\ul\CM'$. To see this first notice that from the above explanation we have
\[
\xymatrix {
\BF_\nu^{r'}\cong\omega^\nu(\ul\CM')/z_\nu\cdot \omega^\nu(\ul\CM') \cong  z_\nu^N \omega^\nu(\ul\CM')/z_\nu^{N+1}\cdot \omega^\nu(\ul\CM')\subseteq \omega^\nu(f)(\omega^\nu(\ul\CM))/z_\nu^{N+1}\cdot\omega^\nu(\ul\CM'),\\
}
\]
and thus one may take $x_1,\dots ,x_{r'} \in \omega^\nu (f)(\omega^\nu (\ul\CM))$ such that they generate the $\BF_\nu$-vector space 
$$
\omega^\nu (\ul\CM')\slash z_\nu \cdot \omega^\nu(\ul\CM').
$$

Let $H:=\sum_{i=1}^{r'}A_\nu\cdot x_i\subseteq\omega^\nu (f) (\omega^\nu (\ul\CM))\subseteq \omega^\nu(\ul\CM')$. As $\omega^\nu(\ul\CM')$ is a free module of rank $r'$, computing the image of $H$ in $\omega^\nu (f) (\omega^\nu (\ul\CM))/z_\nu^{N+1}\omega^\nu(\ul\CM')$, we see that the rank of the free module $H$ also equals $r'$. Therefore from
$$
Q_\nu^{r'}\cong H\otimes_{A_\nu}Q_\nu=\sum_{i=1}^{r'}Q_\nu\cdot x_i\subseteq \omega_{Q_\nu}^\nu (f)(\omega_{Q_\nu}^\nu(\ul\CM))\subseteq\omega_{Q_\nu}^\nu(\ul\CM'),
$$ 
we see that $\dim_{Q_\nu}\omega_{Q_\nu}^\nu (f) (\omega^\nu(\ul\CM))=r'$. 
Now observe that $\rk (\ul\im f)=r'$. To see this, apply $\omega^\nu(-)$ to the morphism $\ul\CM\twoheadrightarrow \ul\im f\subseteq \ul\CM'$, to get a surjection $\omega^\nu(f): \omega^\nu(\ul\CM)\to \omega^\nu(\ul\im f) \subseteq \omega^\nu(\ul\CM')$. Consequently we have 
$$
r'=\dim_{Q_\nu}\omega_{Q_\nu}^\nu (\ul\im f)= \rk(\ul\im f),
$$
\noindent
and therefore $\ul\im f \to \ul\CM'$ is a quasi-isogeny. 

\bigskip
\noindent
$(b) \Rightarrow (c)$ and $(d)$, precisely because of the following commutative diagram

$$
\CD
\omega_Q^\nu(\ul\CM)@>f_\nu>>\omega_Q^\nu(\ul\CM')@>>>0\\
@V{\pi_\nu}VV @VV{\pi_\nu'}V @.\\
\omega_Q^\nu(\ul\CM)@>{f_\nu}>>\omega_Q^\nu(\ul\CM')@>>>0.
\endCD
$$  

\noindent
Now assume that $\pi_\nu$ and $\pi_\nu'$ are semi-simple with characteristic polynomials $\chi_\nu$ and $\chi_\nu'$. Write $\chi_\nu'=\prod_{i=1}^{n'} P_i'$ for irreducible polynomials $P_i'\in Q_\nu[x]$. By semi-simplicity we may write $\omega_{Q_\nu}^\nu(\ul\CM')= \oplus_{i=1}^{n'} V_i'$ as $Q_\nu[\pi_\nu]$-module, where $V_i'\cong Q_\nu[x]/P_i'$, see Remark \ref{Rem_HomDec}. Thus $\chi_\nu=\chi_\nu' \cdot u(x) $ for some $u(x)\in Q_\nu[x]$, and hence $\oplus_{i=1}^{n'} V_i'$ appears as a summand of $\omega_{Q_\nu}^\nu(\ul\CM)$. \\

\noindent
Assume that d) holds and suppose further that $\mu_\pi$ is irreducible, then we see that $\mu_\pi=\mu_\pi'$ and $F=Q[x]/\mu_\pi$ is a field. Therefore by Proposition~\ref{Proph|r} we have $\chi _\nu'=\left(\mu_\pi'\right)^{\rk \ul\CM'/[F:Q]}| \left(\mu_\pi\right)^{\rk \ul\CM/[F:Q]}=\chi_\nu$. Furthermore  $\pi_\nu$ and $\pi_\nu'$ are semi-simple and (b) follows from (c) as in i).\\

\end{proof}

\noindent
This proposition has the following consequence.

\begin{proposition}\label{PropQIsoClassesII}
Keep the notation from the above proposition. Consider the following statements

\begin{enumerate}
\item
$\ul\CM$ is quasi-isogenus to $\ul\CM'$.
\item
There exists an isomorphism $\omega_{Q_\nu}^\nu(\ul\CM)\tilde{\to}\omega_{Q_\nu}^\nu(\ul\CM')$ in $\Hom_{Q_\nu[\Gamma_L]}(\omega_{Q_\nu}^\nu(\ul\CM),\omega_{Q_\nu}^\nu(\ul\CM'))$.
\item
$\chi_\nu=\chi_\nu'$.
\item
$\mu_\pi=\mu_{\pi'}$ and $\rk\ul\CM=\rk\ul\CM'$.
\item
There exist an isomorphism of $Q$-algebras
$$
\alpha:\QEnd_L(\ul\CM)\tilde{\to}\QEnd_L(\ul\CM'),
$$
with $\alpha(\pi)=\pi'$.
\item
There exist an isomorphism of $Q_\nu$-algebras 
$$
\alpha_\nu:\QEnd_L(\omega^\nu(\ul\CM))\tilde{\to}\QEnd_L(\omega^\nu(\ul\CM')),
$$
with $\alpha(\pi_\nu)=\pi_\nu'$.\\

\noindent
Then we have the following statements
\item[i)] $(a)$ and $(b)$ are equivalent and imply $(c)$, $(d)$ and $(e)$. Also $(e)$ implies $(f)$.

\item[ii)] if $\pi_\nu$ and $\pi_\nu'$ are semisimple then we have
$$
(a) \Leftrightarrow (b) \Leftrightarrow (c)\Leftrightarrow(f)\Leftrightarrow(e) 
$$
and $(c) \Rightarrow (d)$.
\item[iii)] if $\mu_\pi$ and $\mu_{\pi'}$ are irreducible in $Q[x]$, then all the above statements are equivalent.

\end{enumerate}
\end{proposition}

\begin{proof}
The statements about $(a)$, $(b)$, $(c)$ and $(d)$ follow from the above Proposition~\ref{PropQIsoClassesI}.\\
\noindent
Precisely ($a$) implies $(e)$. Namely, a quasi-isogeny $f:\ul\CM\to\ul\CM'$  gives the isomorphism  $$\QEnd_k(\ul\CM)\tilde{\to}\QEnd_k(\ul\CM')$$ by sending $g\mapsto f\circ g\circ f^{-1}$. Furthermore we have $\alpha(\pi)=f \circ \pi \circ f^{-1}=f \circ \tau_\CM \circ (\sigma^\ast)\tau_\CM \cdots (\sigma^\ast)^{e-1}\tau \circ f^{-1}$, where using $f\circ\tau_\CM=\sigma^\ast f \circ \tau_{\CM'}$, the later equals $\pi_\CM'$.\\
\noindent 
Suppose $\pi_\nu$ and $\pi_\nu'$ are semisimple. Then the assertion ($f$) implies $(c)$ follows from decomposition $\QEnd_k(\ul\CM)=\oplus_{i=1}^n Mat_{m_\mu\times m_\mu}(K_\mu)$ with $\chi_\nu=\prod_\mu\mu^{m_\mu}$ and $K_\mu=Q_\nu[x]/\mu$; see remark \ref{Rem_HomDec}. \\
\noindent
Suppose $\mu_\pi$ and $\mu_{\pi'}$ are irreducible, then $F=Q[x]/\mu_\pi$ and
$F'=Q[x]/\mu_{\pi'}$ are fields and therefore  $\pi_\nu$ and $\pi_{\nu'}$ are semi-simple. As we have seen above, this implies that $\chi_\nu=\chi_\nu'$. We conclude that $\mu_{\pi}=\mu_{\pi'}$ by Proposition~\ref{Proph|r}(b).

\end{proof}

\noindent
\textbf{-The Grothendieck Ring $K_0(\CM ot_C^{\nu}(\BF) )$}\\

Recall that the category $Cor_\sim(k,\BQ)$ is the category whose objects are smooth projective schemes over $k$ and whose morphisms are given by

$$
\Hom(X,Y)=\oplus_{X_i}Ch_{\sim , \dim X_i}(X_i\times_k Y,\BQ).
$$  
\noindent
Here $X_i$ denote the connected components of $X$ and $Ch_{\sim, d}(-,\BQ)$ denotes the group of cycles of dimension $d$ modulo the equivalence relation $\sim$, with coefficients in $\BQ$. The category $Ch_\sim^{eff}(k)$ is the pseudo-abelian envelop of $Cor_\sim(k,\BQ)$  and the category $Ch_\sim(k,\BQ)$ is obtained by inverting \emph{Lefschetz
motive} $\BL$. The Lefschetz
motive $\BL$ comes from canonical decomposition $[\BP_k^1] = [\Spec k] \oplus \BL$ in $Ch_\sim^{eff}(k)$. When $\sim$ is the numerical equivalence ($\sim=num$) the category $Ch_\sim(k,\BQ)$ is semi-simple abelian category by \cite{Jannsen}. The \emph{Honda-Tate theory} establishes a bijection between the set $\Sigma Ch_{num}(k,\BQ)$ of simple objects of $Ch_{num}(k,\BQ)$ and the conjugacy classes in $\Gal(\BQ^\alg/\BQ) \backslash W(q)$, where $W(q)$ is the subgroup of Weil q-numbers in $(\BQ^\alg)^\times$. Here $q:=\# k$. \\

We now discuss the analogous picture over function fields. Set $W_\ul\nu=\{\alpha\in Q^\alg ; \nu(\alpha)= 0 ~\forall~ \nu\notin \ul\nu  \}$. Consider the free $\BZ$-module $\BZ[\Gamma_Q\backslash W_\ul\nu\times \BN_{\geqslant 1}\slash\sim]$ generated by the equivalence classes in $\Gamma_Q\backslash W_\ul\nu\times \BN_{\geqslant 1}\slash\sim$. Here $(\alpha,n)$ and $(\beta,m)$ are equivalent if $\alpha^{m.l}=\beta^{n.l}$  for some integer $l\in \BN_{\geqslant 1}$. The operation $(\alpha,1)\cdot (\beta,1) = (\alpha\beta, 1)$ induces a ring structure on $\BZ[\Gamma_Q\backslash W\times \BN_{\geqslant 1}\slash\sim]$. Note that to compute $(\alpha,n)\cdot (\beta,m)$, using the equivalence relation, one can adjust both factors to get a product of the form $(\ast,1)\cdot (\ast,1)$. Namely, write $(\alpha,n)= (\alpha',1)$ with $\alpha'^n=\alpha$ (resp. $(\beta,m)= (\beta',1)$ with $\beta'^m=\beta$) and thus we have $(\alpha,n)\cdot(\beta,m)=(\alpha'\beta',1)$. One can easily see that this is independent of the choice of $\alpha'$ and $\beta'$.\\

\begin{theorem}\label{Thm_HT_C-Mot}
There is a bijection 

$$
\text{set $\Sigma$ of simple objects in $\mot(\ol{\BF}_q)$} \leftrightarrow \text{elements of $\Gamma_Q\backslash W_\ul\nu\times \BN_{\geqslant 1}\slash\sim$}.
$$

\end{theorem}

\begin{proof}
Let $\ul\CM:=(\CM,\tau_\CM)$ be a simple object in $\mot(\ol\BF_q)$. Suppose that it comes by base change from a $C$-motive in $\mot(L)$ for a finite extension $L/\BF_q$ of degree $n$, see Proposition \ref{Prop_GeomMixedareMixedOverAFiniteExtensionII}. Let $\pi:=\pi_\ul\CM$ denote the corresponding Frobenius isogeny and let $\mu_\pi$ denote the corresponding minimal polynomial. Let $\alpha_\pi$ be a zero of the minimal polynomial $\mu_\pi$. Then sending $\ul\CM$ to the pair $(\alpha,n)$ gives an assignment $\Sigma\to \Gamma_Q\backslash W_\ul\nu\times \BN_{\geqslant 1}\slash\sim$. This is one to one by Proposition \ref{PropQIsoClassesI}. The fact that it is onto was proved in \cite[Theorem 3.12]{FelixThesis}.

\end{proof}

\begin{remark}
Note that a sketchy proof of the one-to-one part is also given in \cite[Corollary~4.2]{FelixThesis}.

\end{remark}

\begin{corollary}
There is a morphism
$$
K_0(\mot(\ol\BF_q))\to \BZ[\Gamma_Q\backslash W_\ul\nu\times \BN_{\geqslant 1}\slash\sim]
$$
of rings.
\end{corollary}

\noindent
\section{The Zeta-Function}\label{Section_Zeta-Function and analogous motives}
\noindent

Recall that assigning a zeta function to a variety over a finite field $L$, factors through the Grothendieck ring
$$
K_0(Ch_{num}(k,\BQ))\to 1+t\BZ\dbl t\dbr.
$$
The zeta function satisfies sort of properties manifested in Weil conjectures. A crucial observation to prove these conjectures was to establish a cohomology theory for schemes and expressing the zeta function of $X$ in terms of the action of Frobenius on the corresponding cohomology groups.\\

Let us briefly explain the analogous picture over function fields. Let us fix a place $\nu$ away from characteristic places $\nu_i$. In contrast with the above assignment, we define the zeta function associated to a $C$-motive $\ul\CM$ in $\mot(L)$ by the following formula

$$
Z(\ul\CM,t):=\prod_i \det(1-t \pi_\nu |\Koh_{\text{\'et}}^i(\ul\CM,Q_\nu))^{(-1)^{i+1}}.
$$
\noindent
According to Proposition \ref{Prop_RationalTateisexactII}, this assignment defines a morphism $K_0(\mot (L))\to 1+Q_\nu [t]$, which can be shown that in fact factors through $Q\dbl t\dbr$ and gives  

$$
Z(-,t): K_0(\mot (L))\to 1+ t Q \dbl t\dbr,
$$
i.e. the definition is independent of the choice of the place $\nu$. We further define $Z(\ul\CM,t):=\prod_p Z(i_p^\ast \ul\CM, t)$ for a motive $\ul\CM \in \mot(S)$ over a general base scheme $S$, which is of finite type over $\BF_q$. Here the product is over all closed points $i_p: \Spec\kappa(p)\to C$. Let us explain the reason behind the fact that these definitions are independent of the chosen place $\nu$. \\
First assume that $\ul\CM$ is simple. Then $E:=\QEnd(\ul\CM)$ is a central simple algebra over the field $F:=F(\pi)$. For a semi-simple element $f\in E$ we let $\FJ$ denote the commutative subalgebra of rank $r:=\rk \ul\CM$ containing $f$. Consider the norm function $N: E\to Q$ which sends $g$ to $N_{K/Q}(\det(\alpha(g)))$, here $K$ is a splitting field for $E$ and $\alpha: E\otimes_F K \tilde{\to} Mat_{n\times n}(K)$ is an isomorphism. One can see the norm $N(f)$, as the determinant of the $Q$-endomorphism of $\FJ \otimes_F K$ given by multiplication by $f$. Note that one can identify $\FJ_\nu:= \FJ \otimes_Q Q_\nu$ with $\omega_{Q_\nu}^\nu(\ul\CM)$; see \cite[Lemma~7.2]{Bor-Har1}. Therefore $N(f)=\det(\omega^\nu(f))$. 

The above defined norm induces $N(-): \QEnd(\ul\CM)\to Q$ for semi-simple $\ul\CM$, for which the equality $N(f)=\det(\omega^\nu(f))$ holds. Now, to see that this equality holds for general element $f\in \QEnd(\ul\CM)$, we write $\omega^\nu(f)$ in Jordan normal form $S+N$ over $Q_\nu^{alg}$,  and take a power $q^N$ such that $(N)^{K^q}=0$. So $f^{q^N}$ is semi-simple and from the above arguments we see that $N(f^{q^N})=\det (\omega^\nu(f))^{q^N}$ and thus $N(f)=\det(\omega^\nu(f))$. Now for every $a\in A$ we have

$$
\chi_\nu(a)=\det(a\cdot Id - \pi_\nu)=N(a-\pi),
$$
\noindent
and thus the characteristic polynomial $\chi_\nu$ is independent of the chosen place $\nu$. This implies that the zeta function $Z(\ul\CM,t)$ lies in $Q(t)$.

\begin{remark}[zeta function of a $G$-shtuka]\label{Rem_zeta function of a $G$-shtuka}
Let $\ul\CG$ be in $\nabla_n\scrH^1(C,\FG)(L)$. Then to any representation $\rho$ we can assign the zeta function of $\rho_\ast\ul\CG\in \mot(L)$, this gives
$$
[Z(\ul\CG,t)]: R(G)\to Q(t), 
$$
which assigns a rational function to a given  class of representation in the Grothendieck ring of representation $R(G)$.

\end{remark}

\begin{remark}
Assume that $\ul\nu:=(0,\infty)$ for two specified places 0 and $\infty\in C$.  Let $\zeta$ denote the image of the uniformizar of $\CO_{C,0}$ in $L$. One say's that $\ul\CM^\flat \in \mot (\BF_q)$ is analog to $\CM \in DM_{-}^{eff}(\BF)$ if  $Z(\ul\CM,t)$ is the reduction of a lift of the zeta function $Z(\CM, t)$ to $\BZ\dbl \mathsf{y},T \dbr$, regarding the following diagram

\[
\xymatrix {
\mot (\BF)\ar[r] &K_0(\mot (\BF))\ar[r]^{~~~~~Z(-,t)}& A \dbl t\dbr&\\
& &  \BZ \dbl \mathsf{y}, t\dbr\ar[d]^{\mathsf{y}=q}\ar[u]_{\mathsf{y}=z-\zeta}&\\
DM_{-}^{eff}(\BF)\ar[r] & K_0(DM_{-}^{eff}(\BF))\ar[r]^{~~~~~Z(-,t)} \ar[ur]& \BZ\dbl t\dbr&
}
\]

\noindent
For example Carlitz module and $M(\BG_m)$ and supersingular Drinfeld module of rank 2 and ``some" elliptic curve $E$ are analog.

\end{remark}

\section{Semi-simplicity Of The Category Of C-Motives over finite fields}\label{Sect_semisimplicity}

Consider the category $\mot(\ol\BF_q)$. This  is a tannakian category with a fiber functor   
$$
\omega:\mot(\ol\BF_q) \to Q.\ol\BF_q\text{-vector spaces}.
$$
This category admits \'etale and crystalline  realizations. Note that according to Proposition \ref{Prop_GeomMixedareMixedOverAFiniteExtensionII} we may regard the tannakian category $V_\nu(Q_\nu)$ of germs
of $Q_\nu$-adic representation of $\Gal(\ol\BF_q/\BF_q)$ as the \'etale realization category

$$
\omega_{Q_\nu}^\nu(-):\mot(\ol\BF_q)\to V_\nu(Q_\nu)
$$
Recall that this category consists of  equivalence classes of continuous semisimple representations of open subgroups $U$ of $\Gal(\ol\BF_q/\BF_q)$ on the same finite dimensional $Q_\nu$-vector spaces $V$. Where we say
that $\rho_1$ and $\rho_2$ on open subgroups $U_1$ and $U_2$ are equivalent if they agree on an open subgroup of $U_1 \cap U_2$.

\begin{theorem}\label{ThmPoincare-Weil}
The category $\mot(\ol\BF_q)$ with the fiber functor $\omega$, is a semi-simple tannakian category. In particular the kernel $P:=\FP^{\Delta}$ of the corresponding motivic groupoid $\FP:=\Aut^\otimes\bigl(\omega|\mot(\ol\BF_q)\bigr)$ is a pro-reductive group. 
    
\end{theorem}

\begin{proof}

Since a given motive $\ul \CM \in
\mot(\ol\BF_q)$ is defined by data of finite type we may suppose that $\ul \CM$ comes from a motive over
a finite extension $L/\BF_q$, which we again denote by $\ul \CM$.\\
It is enough to show that after a finite extension $L\subset L'\subset
\BF$, the image $\ul \CM'$ of $\ul \CM$ under the obvious functor $\mot(L)\to \mot(L')$ is semi-simple, or
equivalently the endomorphism algebra $E:=\QEnd(\ul \CM')$ is a semi-simple algebra over $Q$; see Theorem \ref{Thm_SemisimplicityQEnd}.\\
Let $\hat{\ul M}_\nu=(\hat{M}_\nu,\tau_\nu)$  denote the $\hat{\sigma}$-crystal associated to $\ul \CM$, for a place $\nu$ distinct from the characteristic places $\nu_i$. It is enough to show that the endomorphism algebra $E_\nu :=E\otimes Q_\nu=\QEnd(\omega^\nu(\ul\CM'))$ is semi-simple. Note that the last equality follows from Theorem \ref{Thm_Tate_Conj}. \\
We can equivalently show that $Q_\nu(\pi_\nu')$ is semi-simple, where $\pi_\nu'$ is the Frobenius endomorphism $\pi_\nu':=\omega^\nu(\pi_{\ul\CM'})$.
To show this take a representative matrix $B_{\pi_\nu}$ for $\pi_\nu\otimes 1 \in End(\hat{\ulM}_\nu\otimes_{A_\nu}
 Q_\nu^{alg})$ and write $B_{\pi_\nu}$ in
the Jordan normal form, i.e. $B=S+N$ where $S$ is semi-simple and $N$ is nilpotent and $SN=NS$.
We take $L'/L$ to be a field extension such that $[L':L]$ is a power of the characteristic of $\BF_q$ and $[L':L]\ge \rank M$. Clearly $B_{\pi_\nu}^{[L':L]}=(S+N)^{[L':L]}=S^{[L':L]}+N^{[L':L]}=S^{[L':L]}$ is semi-simple. This represents the Frobenius endomorphism $\pi'\otimes 1$ in $E_\nu\otimes Q_\nu^{alg}$. Since $Q_\nu^{alg} / Q$ is perfect we may argue by \cite[Proposition 9.2/4]{BourbakiAlgebra} that
$\pi_\nu'$ is semi-simple, and as we mentioned above, this suffices.\\
The second Part of the proposition follows from \cite[Proposition 2.23]{De-Mi}.

\end{proof}

\bigskip

%%%%%%%%%%%%%%%%%%%%%%%%%%%%%%%%%%%%%%%%%%%%%%%%%%%%%%%%%%%%%%%%%%%%%%
%
%    Bibliography
%
%%%%%%%%%%%%%%%%%%%%%%%%%%%%%%%%%%%%%%%%%%%%%%%%%%%%%%%%%%%%%%%%%%%%%%

\vfill

\begin{minipage}[t]{0.5\linewidth}
\noindent
Esmail Arasteh Rad\\
Universit\"at M\"unster\\
Mathematisches Institut \\
Einsteinstr.~62\\
D -- 48149 M\"unster
\\ Germany
\\[1mm]
\end{minipage}
\begin{minipage}[t]{0.45\linewidth}
\noindent
Urs Hartl\\
Universit\"at M\"unster\\
Mathematisches Institut \\
Einsteinstr.~62\\
D -- 48149 M\"unster
\\ Germany
\\[1mm]
\end{minipage}

\end{document}